\documentclass[11pt]{article}

\usepackage{amsmath,amssymb,amsthm,amscd,epsfig}
\newcommand{\vp}{\varepsilon}

\newcommand{\n}{\noindent}

\newcommand{\ovl}{\overline}

\newcommand{\bb}[1]{\mathbb{#1}}

\theoremstyle{plain}

\newtheorem{thm}{Theorem}
\newtheorem*{unthm}{Theorem}

\newtheorem{cor}{Corollary}

\newtheorem{pro}{Proposition}

\theoremstyle{definition}

\newtheorem{defn}{Definition}

\theoremstyle{remark}

\begin{document}

\title{Inverse pressure estimates and the independence of stable dimension
for non-invertible maps}

\author{Eugen Mihailescu and Mariusz Urba\'nski}

\date{}
\maketitle

\begin{abstract}

We study the case of an Axiom A holomorphic non-degenerate (hence non-invertible) map $f:\bb P^2
\bb C \to \bb P^2 \bb C$, where $\bb P^2 \bb C$ stands for the complex 
projective space of dimension 2. Let $\Lambda$ denote a basic set for $f$ of 
unstable index 1, and $x$ an arbitrary point of $\Lambda$; we denote by $\delta^s(x)$ the Hausdorff dimension of $W^s_r(x) \cap \Lambda$, where $r$ is some fixed positive number and $W^s_r(x)$ is the local stable manifold at $x$ of size $r$; $\delta^s(x)$ is called \textit{the stable dimension at} $x$.
In \cite{MU2}, Mihailescu and Urba\'nski introduced a notion of inverse 
topological pressure, denoted by $P^-$, which takes into consideration preimages of points. In \cite{MM}, Manning and McCluskey study the case of hyperbolic diffeomorphisms on real surfaces and give formulas for Hausdorff dimension. Our non-invertible situation is different here since the local unstable manifolds
are not uniquely determined by their base point, instead they depend in general
on whole prehistories of the base points. 
Hence our methods are different and are based on using a sequence of inverse pressures for the iterates of $f$, in order to give upper and lower estimates of the stable dimension (Theorem \ref{thm1}). As a Corollary, we obtain an estimate of the oscillation of the stable dimension on $\Lambda$. 
When each point $x$ from $\Lambda$ has the same number $d'$ of preimages in $\Lambda$, then we show in Theorem \ref{thm2} that $\delta^s(x)$ is independent of $x$; in fact $\delta^s(x)$ is shown to be equal in this case with the unique zero of the map $t \to P(t\phi^s - \log d')$.
We also prove the Lipschitz continuity of the stable vector spaces over
$\Lambda$; this proof is again different than the one for diffeomorphisms (however, the unstable distribution is not always Lipschitz for conformal non-invertible maps). In the end we include the corresponding results for a real conformal setting.  

\end{abstract}

\bigskip

\n {\bf Keywords:} \ Hausdorff dimension, stable manifolds, basic sets, 
inverse topological pressure.\bigskip

\n \S 1. Introduction and notations. Inverse topological pressure.

\n \S 2. Estimates from above and below for the stable dimension in the general holomorphic case using the inverse pressure of iterates.

\n \S 3. Independence of $\delta^s(x)$ (of $x$), when the map $f$ is open on $\Lambda$.

\n \S 4. Results in the real conformal case.

\bigskip


\textbf{Acknowledgement:} The first author was supported in part by grant CNCSIS nr.153 (2004-2005), from the Romanian Ministry of Education and Research.

\newpage

\section{Introduction and notations. Inverse topological pressure}\label{sec1}

In the case of $C^2$ Axiom A diffeomorphisms of real surfaces, Manning and McCluskey (\cite{MM}) proved that the Hausdorff dimension of a basic set $\Lambda$ is given by the formula $HD(\Lambda) = \delta_u + \delta_s$, with $\delta_u, \delta_s$ being the unique zeros of the pressure functions of the potentials $ - t \log |Df_u|, t\log |Df_s|$ respectively, considered on $\Lambda$.
For the case of hyperbolic automorphisms on $\mathbb C^2$ (Henon maps), Verjovsky and Wu (\cite{VW}) showed that the Hausdorff dimension of the intersection between local stable manifolds 
and the Julia set is given also as the unique zero of a pressure function. For non-invertible conformal maps $f$ (for example holomorphic maps on the projective complex space $\mathbb P^2$) which are hyperbolic on a basic set $\Lambda$, the situation is completely different, and as shown in \cite{Mi} and \cite{MU1}, this stable dimension (precise definition will be given later) is not equal to the unique zero of the pressure function. At the same time, we do not have a uniquely determined unstable manifold going through a given point of the basic set $\Lambda$.
In order to deal with the non-invertible case, Mihailescu and Urbanski have introduced a notion of inverse pressure (\cite{MU2}), which takes into consideration all the inverse iterates of points (instead of the forward iterates from the case of usual topological pressure). In this paper we will obtain a theorem (Theorem \ref{thm1}) giving lower estimates of the stable dimension 
by using zeros of inverse pressures of iterates of $f$. As a Corollary we obtain an estimate of the maximum possible oscillation of the stable dimension on $\Lambda$. 

Then, when the map is open on the basic set $\Lambda$, we will prove (Theorem \ref{thm2}) that the stable dimension is independent of the point; in the proof we use again ideas and concepts related to inverse pressure. 

Most of these proofs and results work for a more general setting (finite-to-one conformal maps with hyperbolic structure on a basic set, and with the dimension of the stable vector spaces equal to 2), but we preffer to state them firstly in the case of holomorphic Axiom A maps on $\mathbb P^2$, and we include a section at the end of the paper with the theorems in the more general case. 

Note also that in Theorem \ref{thm0} we actually use the holomorphicity at the end of the proof.
 
In this section we recall some definitions and properties of inverse pressure,
which will be used later. We consider the following setting:

$X$ is a compact metric space, $f: X \to X$ is a continuous surjective map on $X$, and $Y\subseteq X$ is a subset of $X$.
Due to the surjectivity of $f$, for any point $y$ of $X$, and any positive integer $m$, there exists $y_{-m} \in X$ such that $f^m (y_{-m}) = y$.
By \textit{prehistory of length} $m$ (or $m$-\textit{prehistory}, or 
\textit{branch of length} $m$) of $y$,
we will understand a collection of consecutive preimages of $y$, $C = 
(y, y_{-1},..., y_{-m})$, where $f(y_{-i}) = y_{-i+1}, i=1,..,m, y_0=y$.
Given a prehistory $C$, we shall denote by $n(C)$ its length. Fix $\vp>0$.
Denote by $\mathcal{C}_m$ the set of all $m$-prehistories of points from $X$.
For such an $m$-prehistory $C$, let $X(C, \vp)$ be the set of points 
$\vp$-\textit{shadowed} by $C$ (in backward time) i.e:
$X(C,\vp):= \{z \in B(y_0,\vp): \exists 
z_{-1}\in f^{-1}(z) s.t.\ d(z_{-1}, y_{-1})<\vp,.., \exists z_{-m} \in 
f^{-1}(z_{-m+1}) s.t. \ d(z_{-m}, y_{-m}) < \vp\}$.
Given the $m$-prehistory of $y$, $C= (y, y_{-1}, ..., y_{-m})$ and a real 
continuous function $\phi$ on $X$, (we denote the set of real continuous 
functions on $X$, by $\mathcal{C}(X, \bb R) $), one can define the \textit{consecutive sum} of 
$\phi$ on $C$, 
$S^-_m \phi (C)= \phi(y) + \phi(y_{-1}) +...+ \phi(y_{-m})$.
We may also use the notation $S^-_m \phi(y_{-m})$ instead of $S^-_m \phi(C)$.
We will define now the inverse pressure $P^-$ by a procedure similar to that used in the case of Hausdorff outer measure.
Let $\phi$ be an arbitrary continuous function, $\phi \in \mathcal{C}(X, 
\bb R)$; let also $\lambda$ a real number and $N$ a positive integer.
Denote by $\mathcal{C}_* := \mathop{\cup}\limits_{m \ge 0} \mathcal{C}_m$.
We say that a subset $\Gamma \subset \mathcal{C}_*$, $\vp$-\textit{covers} $X$
if $X = \mathop{\cup}\limits_{C \in \Gamma} X(C, \vp)$.
Then define the following expression
$$
\aligned
M^-_f(\lambda, \phi, & Y, N, \vp):= \inf \{\sum_{C \in \Gamma} exp(-\lambda
n(C) + S^-_{n(C)}\phi(C)), n(C) \ge N, \forall C \in \Gamma, \\
& \text{and} \ \Gamma 
\subset \mathcal{C}_* \ \text{s.t} \
 Y\subset \mathop{\cup}\limits_{C\in \Gamma} X(C, \vp) \}
\endaligned
$$
When $N$ increases, the set of acceptable candidates $\Gamma$ which 
$\vp$-cover  X gets smaller , therefore the infimum increases in the previous
expression.
Hence $\lim\limits_{N \to \infty} M^-_f(\lambda,\phi,Y,N,\vp)$ exists and will
be denoted by $M^-_f(\lambda,\phi,Y,\vp)$.
Now, let $P^-_f(\phi, Y, \vp) := \inf \{\lambda:M^-_f(\lambda, \phi,Y,
\vp) = 0\}$. Consider two positive numbers $\vp_1 < \vp_2$ and let us compare 
$P^-_f(\phi, Y, \vp_1)$ and $P^-_f(\phi, Y, \vp_2)$.
Given any prehistory $C$, we have that $X(C, \vp_1) \subset X(C, \vp_2)$, so if $\Gamma \subset \mathcal{C}_*$ $\vp_1$-covers $Y$, then $\Gamma$ also 
$\vp_2$-covers $Y$.
Therefore there are more candidates $\Gamma$ in the expression of $M^-_f(\lambda, \phi, Y, N, \vp_2)$ than in the expression of $M^-_f(\lambda, \phi, Y, N, \vp_1)$. This shows that for any $N$,
$M^-_f(\lambda,\phi,Y,N,\vp_2) \le M^-_f(\lambda, \phi, Y,N,\vp_1)$.
Hence $0 \le M^-_f(\lambda,\phi,Y,\vp_2) \le M^-_f(\lambda,\phi,Y,\vp_1)$, 
and then from definition, $P^-_f(\phi,Y,\vp_2) \le P^-_f(\phi,Y,\vp_1)$.
This proves that, when $\vp$ decreases to 0, $P^-_f(\phi,Y,\vp)$ increases, so 
the limit $\lim\limits_{\vp\to 0}
P^-_f(\phi,Y,\vp)$ does exist and is denoted by $P^-_f(\phi,Y)$.
$P^-_f(\phi,Y)$ is called \textit{the inverse pressure} 
(or \textit{inverse upper pressure}) of $\phi$ on $Y$. $P^-_f(\phi,Y,\vp)$
is called the $\vp$-\textit{inverse pressure} of $\phi$ on $Y$.
This notion has been introduced in \cite{MU2}; here we have used slightly
different notations.
When the map $f$ will be clear from the context, we may drop the index $f$ 
from the notations for $P^-_f(\phi, Y), P^-_f(\phi,Y,\vp), M^-_f(\lambda,
\phi,Y,N,\vp)$, etc.
Also, we will denote by $P^-_f(\phi), P^-_f(\phi,\vp), M^-_f(\lambda,\phi,
N,\vp)$, etc., the quantities $P^-_f(\phi,X), P^-_f(\phi,X,\vp)$, 
$M^-_f(\lambda,\phi,X,N,\vp)$, etc., respectively.
The following proposition provides some properties of $P^-$.

\begin{pro}\label{pro1}
Let $f:X \to X$ be a continuous surjective map on the compact metric 
space $X$, $\vp$ a positive number and $\phi$ a function from 
$\mathcal{C}(X, \bb R)$. 
\begin{itemize}
\item[i)] If $Y_1\subset Y_2\subset X$, then $P^-_f(\phi,Y_1) \le P^-_f(\phi,
Y_2) $ and $P^-_f(\phi,Y_1,\vp) \le P^-_f(\phi,Y_2,\vp)$.
\item[ii)] If $Y = \mathop{\cup}\limits_{j\in J} Y_j$ is a finite or 
countable union 
of subsets of $X$, then $P^-_f(\phi,Y,\vp) = \sup\limits_{j \in J}
P^-_f(\phi,Y_j,\vp)$ and $P^-_f(\phi,Y) = \sup\limits_{j \in J} 
P^-_f(\phi,Y_j)$.
\item[iii)] If $f$ is a homeomorphism on $X$, then $P^-_f(\phi) = P_f(\phi)$,
where $P_f(\phi)$ denotes the usual (forward) topological pressure of 
$\phi$ with respect to the map $f$.
\item[iv)] $P^-_f(\phi,Y)$ is invariant to topological conjugacy, i.e 
if $f:X \to X, g:X' \to X'$ are continuous surjective maps and $\Psi:X\to X'$
is a homeomorphism such that $\Psi \circ f = g \circ \Psi$, then $P^-_f(\phi,
Y) = P^-_g(\phi\circ \Psi^{-1}, \Psi(Y))$, for any subset $Y \subset X$.
\end{itemize}
\end{pro}

\begin{proof}
We will prove only part ii), the others are straightforward.
Assume that $Y=\mathop{\cup}\limits_{j\in J} Y_j$ is a finite or countable 
union of 
subsets  of $X$. We will show that, given some $\vp >0$, $P^-_f(\phi,Y,\vp)=
\sup\limits_j P^-_f(\phi,Y_j, \vp)$, for any function $\phi \in \mathcal{C}(X,
\bb R)$; the other equality, $P^-_f(\phi, Y) = \sup\limits_j P^-_f(\phi,Y_j)$
will follow similarly.
First, directly from the definition of $P^-$, it follows that $P^-_f(\phi, Y,
\vp) \ge \sup\limits_j P^-_f(\phi,Y_j,\vp)$.
Take now $t > \sup\limits_j P^-_f(\phi,Y_j,\vp)$. Then there exists some
number $\alpha >0$ so small that  $t-\alpha > P^-_f(\phi, Y_j, \vp), \forall
j \in J$.
So $M^-_f(t-\alpha,\phi,Y_j,\vp)=0$ for all $j\in J$. But from the fact that
$M^-_f(t-\alpha,\phi,Y_j,N,\vp)$ grows with $N$, we obtain that 
$M^-_f(t-\alpha, \phi, Y_j,N,\vp) = 0, \forall j\in J, \forall N >0$.
So, if $N$ is fixed, then for any $j \in J$ there exists a set 
$\Gamma_j \subset \mathcal{C}_*$ such that $Y_j \subset 
\mathop{\cup}\limits_{C\in \Gamma_j} X(C, \vp)$ and
$n(C) \ge N, \forall C\in \Gamma_j$ and we have 
$\sum_{C\in \Gamma_j} exp(-(t-\alpha)n(C) + S^-_{n(C)}\phi(C)) \le 
\frac{1}{2^j}$.
Now, if we consider the collection $\Gamma:=\mathop{\cup}\limits_{j \in J}
\Gamma_j$, then $Y= \mathop{\cup}\limits_{j \in J}Y_j \subset \mathop{\cup}
\limits_{C\in \Gamma}X(C,\vp), n(C) \ge N, \forall C \in \Gamma$, and
$\sum_{C\in \Gamma} exp(-(t-\alpha)n(C) + S^-_{n(C)}\phi(C)) \le 1$.
This means that $M^-_f(t-\alpha,\phi,Y,N,\vp)\le 1$, hence 
$M^-_f(t,\phi,Y,N,\vp) \le e^{-\alpha N}$. Thus $M^-_f(t,\phi,Y,\vp) = 0$ 
and $t\ge P^-_f(\phi,Y,\vp)$.
In conclusion, since $t$ has been taken arbitrarily larger than 
$\sup\limits_{j \in J} P^-_f(\phi, Y_j,\vp)$,  we obtain the required 
equality, 
$P^-_f(\phi,Y,\vp) = \sup\limits_{j \in J} P^-_f(\phi,Y_j,\vp)$.
\end{proof}

Here are also some additional properties of $P^-$, whose proofs can 
partly be found in \cite{MU2}; the proofs of the properties for 
$\vp$-inverse pressures are similar.

\begin{pro}\label{pro2}
Let $f:X\to X$ be a continuous surjective map on the compact metric 
space $X$, $Y$ a subset of $X$ and $\phi, \psi \in \mathcal{C}(X,\bb R)$. 
Then:
\begin{itemize}
\item[i)] $P^-_f(\phi+\alpha,Y) = P^-_f(\phi,Y)+\alpha.$
\item[ii)] If $\phi \le \psi$ on $Y$ and $\vp$ is a positive number, then 
$P^-_f(\phi,Y) \le P^-_f(\psi,Y)$ and $P^-_f(\phi,Y,\vp) \le P^-_f(\psi,Y,\vp).$
\item[iii)] $P^-_f(\cdot, Y)$ is either finitely valued or constantly $\infty$.
\item[iv)] $|P^-_f(\phi,Y) - P^-_f(\psi,Y)| \le ||\phi-\psi|| $ if 
$P^-_f(\cdot,Y)$ is finitely valued; a similar inequality holds for the 
corresponding $\vp$-inverse pressures.
\item[v)] $P^-_f(\phi+\psi\circ f - \psi, Y) = P^-_f(\phi,Y)$.
\item[vi)] If $\phi$ is a strictly negative function on $X$, then the mapping 
$t \to P^-_f(t\phi, Y)$ is strictly 
decreasing if $P^-_f(\cdot,Y)$ is finitely valued. Also the mapping
$t \to P^-_f(t \phi, Y, \vp)$ is strictly decreasing.
\end{itemize}
\end{pro}

The inverse entropy $h^-$ obtained by definition as $P^-(0)$ is 
smaller or equal than 
the preimage entropy $h_i$ (\cite{MU2}) and actually, in the case of 
homeomorphisms, they both coincide with the usual topological entropy 
(definitions and useful properties of $h_i$ 
are given, for example, in \cite{N}, \cite{MU2}, etc).
Another interesting property of $P^-$ gives an alternative way of calculating
the inverse pressure, which will be used in a proof later:

\begin{pro}\label{pro3}[\cite{MU2}]
Let $f:X \to X$ be a continuous surjective map on a compact metric space $X$, 
and $\phi \in \mathcal{C}(X, \bb R)$. 
Denote by 
$Q_m^- (\phi, \vp):= \inf \{\sum\limits_{C\in \Gamma} exp(S^-_m \phi(C)), \Gamma
\subset \mathcal{C}_m, \Gamma \ \vp-\text{covering} \ X\}$.
Then $P^-(\phi) = \lim\limits_{\vp \to 0}\ovl{\lim\limits_{m \to \infty}}
\frac{1}{m} \cdot \log Q_m^- (\phi,\vp)$.
\end{pro}

In the sequel, we will focus on the case of a holomorphic non-degenerate map 
$f:\bb P^2 \to \bb P^2$, where $\bb P^2$ represents the 2-dimensional complex 
projective space $\bb P^2 \bb C$. Any holomorphic map $f$ on $\bb P^2$ is 
given 
as  $ f([z:w:t]) = [P(z,w,t):Q(z,w,t):R(z,w,t)]$, with $P, Q, R$ homogeneous 
polynomials in $z,w,t$, all having the same degree $d$.
If $d \ge 2$, then $f$ is called non-degenerate; in this case $f$ is 
non-invertible.
We shall assume in the sequel that $f$ is non-degenerate and has Axiom A; 
let $\Lambda$ be one of its basic sets of \textit{unstable index} 1, meaning 
that $Df$ has on $\Lambda$
both stable and unstable directions. For definitions and discussions of 
Axiom A for non-invertible maps \cite{Ru} or \cite{Mi} 
are good references.
An important point to remember is that, since $f$ is not invertible on the invariant set $\Lambda$, one has to define hyperbolicity with respect to the 
\textit{natural extension} of $\Lambda$.
We recall briefly this notion and also how to define hyperbolicity in this non-invertible case.
Denote first by $\hat \Lambda := \{\hat x = (x, x_{-1},...) \text{where} \
 x_{-i} \in \Lambda \ \text{and} \ f(x_{-i-1})=x_{-i}, i \ge 0, x_0 = x \}$ 
and call this set the \textit{natural extension} of $\Lambda$ with 
respect to $f$.
$\hat \Lambda$ is a compact metric space endowed with the metric 
$d(\hat x,\hat y) = \sum_{i\ge 0} \frac{d(x_{-i}, y_{-i})}{2^i}$.
More general, we can define a metric $d_K$ on $\hat \Lambda$ for any 
$K > 1$ by setting $d_K(\hat x, \hat y) = \sum_{i \ge 0} \frac{d(x_{-i}, y_{-i})}{K^i}.$
As above, we will not specify the constant $K$ in the notation 
$d_K$ when $K=2$.
Also, it can be noticed that for all $K > 1$, $d_K$ gives the same 
topology on $\hat \Lambda$, namely the topology induced on the subset
$\hat \Lambda$ by the product topology on the larger space $\Lambda^{\bb N}$.
We denote by $\pi:\hat \Lambda \to \Lambda$ the canonical projection 
$\pi(\hat x) = x$ and by $\hat f$ the homeomorphism $\hat f: \hat \Lambda 
\to \hat \Lambda, \hat f (\hat x) = (fx, x, x_{-1}, ...)$. 
The hyperbolicity of $f$ on $\Lambda$ means that there exist
constants $C >0, \lambda' >1$, and for every $\hat x \in \hat \Lambda$, 
a vector space $E^u_{\hat x} \subset T_x \bb P^2$, and a vector space 
$E^s_x \subset T_x \bb P^2$ such that $Df(E^u_{\hat x}) \subset 
E^u_{\widehat{fx}}, Df(E^s_x)
\subset E^s_{fx}$ and we have the inequalities
$||Df^k_x(v)|| \le C  (\lambda')^{-k}||v||, ||Df^k_x(w)|| \ge C 
(\lambda')^k ||w||$,
for every $x\in \Lambda, k \ge 0$ and all vectors $v\in E^s_x, 
w\in E^u_{\hat x}$. 
In the definition of hyperbolicity on $\hat \Lambda$ we assume also that 
$E^s_x \oplus E^u_{\hat x} = T_x \bb P^2, \forall \hat x\in \hat \Lambda$ and that $E^s_x$ depends continuously on $x$, while $E^u_{\hat x}$ depends 
continuously on $\hat x$.
$E^s_x$ is called the \textit{stable tangent vector space} (or the 
\textit{stable space}) at $x$. $E^u_{\hat x}$ is called the 
\textit{unstable tangent vector space} (or \textit{unstable space}) 
corresponding to the prehistory  $\hat x$.
Like in the diffeomorphism case, it is possible (\cite{Ru}) to show that, 
if $r$ is small enough (for example $0 < r < r_0$), there exist stable and unstable local manifolds passing through $x$:
$W^s_r(x) := \{y \in \bb P^2, d(f^i x, f^i y) < r, i\ge 0\}$, and $W^u_r(\hat x) := \{y \in \bb P^2, \exists \  \hat y\in\pi^{1}(y) \
\text{with} \ d(y_{-i}, x_{-i}) < r, i \ge 0\}$.
If moreover $f$ is holomorphic on $\bb P^2$, the local (un)stable 
manifolds on a basic set of unstable index 1, are analytic disks.

Now, given a point $x\in \Lambda$ and a small fixed number $0< r < r_0 < \frac{\text{diam}
\Lambda}{2}$, denote by 
$\delta^s(x):= HD(W^s_r (x) \cap \Lambda)$, where $HD$ stands for the 
Hausdorff dimension of a set.
We shall call $\delta^s(x)$, \textit{the stable dimension at} $x$.
In the sequel we shall suppose also that $\mathcal{C}_f \cap \Lambda = 
\emptyset$, where $\mathcal{C}_f$ denotes the critical set of $f$. 
Hence, one can define the negative function $\phi^s(y):= \log |Df|_{E^s_y}|,
y \in \Lambda$; as a notational remark, $E^s_y$ is a one-dimensional 
complex space and $|Df|_{E^s_y}|$ denotes the norm of $Df$ restricted to this 
stable space.

We studied the stable dimension in \cite{Mi}, \cite{MU1}, \cite{MU2}.
In \cite{Mi}, the first author showed that $\delta^s(x) \le t^s_*$, where 
$t^s_*$ is the unique zero of the pressure function $t \to P(t\phi^s)$ (the 
topological pressure being calculated with respect to the map $f|_{\Lambda}$).
However in the above inequality we do not have equality in general.
Indeed the gap between $\delta^s(x)$ and $t^s_*$ is influenced by the number 
of preimages that a point from $\Lambda$ has in $\Lambda$, as was explained in 
\cite{MU1}, where we obtained a better upper estimate $t^s_0$:

\begin{unthm}

In the above setting, assume that the map $f|_{\Lambda}$ has the property that
every point $x \in \Lambda$ has at least $d' \le d$ preimages in $\Lambda$.
Then $\delta^s(x) \le t^s_0$, where $t^s_0$ is the unique zero of the 
function $t \to P(t\log |Df|_{E^s_y}| - \log d')$ and as a consequence,
$\delta^s(x) \le \frac{h(f|_{\Lambda})-\log d'}{|\log \mathop{\sup}\limits_{y \in \Lambda} |Df|_{E^s_y}||}$.

\end{unthm}

Let us focus now on the zeros $t^s_n(\vp)$ of the $\vp$-inverse pressure 
functions 
for the iterates $f^n|_{\Lambda}$.
If $\Lambda$ is a basic set for $f$, then $f(\Lambda) = \Lambda$, hence
$f^n(\Lambda) = \Lambda, \forall n > 0$ integer.
Let us denote by $Df_s(y)$ the linear map $Df|_{E^s_y}$; 
similarly,
$D f^n_s(y)$ denotes $Df^n|_{E^s_y}, y \in \Lambda$.
Since $f$ is conformal on stable manifolds, $|Df^n_s(y)| = |D f_s(y)|\cdot |D f_s(fy)|\cdot...
\cdot |D f_s(f^{n-1}y)|, \forall y \in \Lambda$.
$\phi^s_n(y) := \log |Df^n_s (y)|, y \in \Lambda$, so $\phi^s_n$ is a strictly 
negative function on $\Lambda$, which has finite values since $\mathcal{C}_f
\cap \Lambda = \emptyset$.
From Proposition 2 vi) applied to $f^n|_{\Lambda}:\Lambda\to \Lambda$, it 
follows that the function $t \to P^-_{f^n}(t \phi^s_n,\vp)$ is strictly 
decreasing; since $P^-_{f^n}(0,\vp) \ge 0$, and $P^-_{f^n}(t\phi^s_n,\vp) < 0$ for $t>0$ large enough,
it follows that this strictly decreasing function has a unique zero, denoted by
$t^s_n(\vp)$. The same is true for the function $t \to P^-_{f^n}(t \phi^s_n)$
which has a unique zero $t^s_n$. When $n=1$ we denote $t^s_1(\vp)$ by $t^s(\vp)$, and $t^s_1$ by $t^s$.
We shall prove in the sequel that $t^s_n(\vp) \ge t^s_{np}(\vp)$ and $t^s_n = 
t^s$, for any positive integers $n, p$ and any $\vp >0$.

First, we will prove that the stable spaces $E^s_y$ depend Lipschitz 
continuously on $y \in \Lambda$. In addition we will show the Lipschitz 
continuity of $y \to E^s_y$ when
$y$ ranges in $W^s_r(x)$ ( $x \in \Lambda$), and moreover, that the 
Lipschitz
constant on these stable leaves can be chosen independently of the point 
$x\in \Lambda$ in the holomorphic case. Remark also that the unstable spaces cannot depend Lipschitz on their base points
since in general they depend on whole prehistories. In \cite{Mi}, one of the authors 
showed that the unstable spaces $E^u_{\hat x}$ depend Holder continuously on $\hat x$, with 
respect to a fixed metric 
$d_K$ on $\hat \Lambda$; the respective Holder exponent depends on the chosen constant $K>1$.
The following theorem was known in the case of conformal diffeomorphisms, 
but up to our knowledge it has never appeared in the case of non-degenerate 
holomorphic maps on $\mathbb P^2$ (which are non-invertible). 
As it turns out below, the non-invertible case requires its own proof, different from the one given for diffeomorphisms. 
(for example, in the non-invertible situation we cannot use the inverse 
iterate $f^{-1}$, and on the natural extension $\hat \Lambda$ we cannot use 
a differentiable structure).

\begin{thm}\label{thm0}
Consider $f:\bb P^2 \to \bb P^2$ a holomorphic Axiom A map, and let $\Lambda$
be one of its basic sets of unstable index 1, such that
$\mathcal{C}_f \cap \Lambda = \emptyset$.
Then the map $x \to E^s_x$ is Lipschitz continuous as a map from $\Lambda$ to the bundle
$G_1(\Lambda)$ of spaces of complex dimension 1 in the tangent bundle over
$\Lambda$, i.e. there exists
a positive constant $\Upsilon$ such that for all $x, y$ from $\Lambda$,
$d(E^s_x, E^s_y) \le \Upsilon d(x, y)$.
In particular, if $\phi^s(y):= \log |Df|_{E^s_y}|, y \in \Lambda$, then
$\phi^s$ is Lipschitz continuous.
Moreover, there exist a small $r>0$ and $\Xi>0$ such that for any $x \in
\Lambda$ and any points $y, z \in W^s_r(x)$, we have $|\phi^s(y) - \phi^s(z)|
\le \Xi \cdot d(y,z)$. 
\end{thm}

\begin{proof}
For every $K>1$ consider the metric $d_K$ on $\hat \Lambda$, given by the formula
$d_K(\hat x, \hat y):= d(x,y)+\frac{d(x_{-1},y_{-1})}{K}+\frac{d(x_{-2}, y_{-2})}{K^2}+...$.
Notice that the topology given by $d_K$ on
$\hat \Lambda$ is independent of $K$ and is induced by the product topology
on a countable product of $\Lambda$'s.
In the sequel we shall use a Pointwise H\"{o}lder Section Theorem from \cite{W}.

\begin{unthm}[Pointwise H\"{o}lder Section Theorem]
Let $E = X \times Y$ be a vector bundle over a metric space $X$, where $Y$
is a closed, bounded subset of a Banach space, and let $\pi: E \to X$ be the
canonical projection. Let $F:E \to E$ be a bundle map covering a
homeomorphism $h:X \to X$, i.e $\pi \circ F = h \circ \pi$.
Suppose that $F$ satisfies the following conditions:
\begin{itemize}
\item[1)] $F$ contracts the fibers of $E$ in the sense that, for all $x \in X$
there exists a constant $0 \le \lambda_x < 1$ such that
$d(F(x,y), F(x,z)) \le \lambda_x d(y,z), \forall y, z \in Y.$
\item[2)] There exist constants $L \ge 1$ and $\alpha > 0$ such that for all
$x, x' \in X$ and $y \in Y$, $|F(x,y)-F(x',y)| \le L \cdot d(x,x')^{\alpha}.$
\item[3)] There exists some positive number $\eta$ such that
$\sup\limits_{x\in X} \lambda_x \cdot \mu_x^{-\alpha} =: \rho(\alpha) < 1$
where $\mu_x$ denotes:
$\mu_x:= \inf \{\frac{d(hx, hx')}{d(x,x')}, x, x' \in X, x \ne x', d(x, x')
< \eta \}$.
Also, let us denote by $\mu:= \inf\limits_{x\in X}\mu_x$ and assume that
$\mu > 0$.
\end{itemize}
Then we have the following:
\begin{itemize}
\item[i)] there exists a unique section $\sigma :X \to E$ whose image is
invariant under $F$, i.e $\sigma \circ h(x) = F \circ \sigma(x), x \in X$.
\item[ii)] $\sigma$ is H\"{o}lder continuous with exponent $\alpha$, i.e
$|\sigma(x) - \sigma(x')| \le H d(x,x')^\alpha, \forall x, x' \in X.$
\item[iii)] Assume that the diameter of $Y$ is bounded by R, then we can
bound the H\"{o}lder constant $H$ by:
$H \le \frac{LR}{\mu \eta^\alpha (1-\sup \lambda_x \mu_x^{-\alpha})}$.
\end{itemize}
\end{unthm}
Let us now return to our setting and see how we can apply this theorem.
By definition of hyperbolicity of $f$, there exists a continuous splitting of
the tangent bundle to $\bb P^2$ over $\hat \Lambda$, given by
$T_{\hat \Lambda} \bb P^2 = E^s \oplus E^u$, where $E^s_x$ depends
continuously on $x \in \Lambda$ and $E^u_{\hat x}$ depends continuously on
$\hat x \in \hat \Lambda$.
The stable space $E^s_x$ and the stable manifold of size $r>0$ at $x$ 
depend only on the forward
iterates of $x$, whereas the unstable space $E^u_{\hat x}$ and the unstable 
manifold $W^u_r(\hat x)$ depend
on the entire prehistory $\hat x$ of $x$.
Let us take an arbitrary constant $K > 1$ and consider the metric $d_K$
on $\hat \Lambda$.
Since continuous maps can be approximated by Lipschitz continuous maps,
there exists a splitting $F^s \oplus F^u(K)$ of $T_{\hat \Lambda} \bb P^2$
such that
the linear subspaces of complex dimension 1, $F^s_x$, depend Lipschitz
continuously on $x \in \Lambda$ and the subspaces of dimension 1,
$F^u_{\hat x}(K)$
depend Lipschitz on $\hat x \in \hat \Lambda$; also we assume that
$F^s_x$
approximates $E^s_x$, and $F^u_{\hat x}(K)$ approximates $E^u_{\hat x}$
uniformly in $x$, respectively $\hat x$. As a remark, the spaces
 $F^u_{\hat x}(K)$ depend in general on $K$ since they have to vary Lipschitz 
continuously with respect to the metric $d_K$, whereas the spaces $F^s_x$ are 
Lipschitz only with respect to the usual euclidian metric induced on 
$\Lambda$, therefore they do not depend on $K$. 
 Let us assume that $d(F^s_x, E^s_x)< \vp, d(F^u_{\hat x}(K), E^u_{\hat x}) <
\vp$, for all $\hat x$ in $\hat \Lambda$, where $\vp$ is a small
positive number.
From the above Lipschitz conditions, there exist positive constants $\tau$
and $\tau_K$ such that
$ d(F^s_x, F^s_y) \le \tau d(x,y), \forall x, y \in \Lambda$,
and $d(F^u_{\hat x}(K), F^u_{\hat y}(K)) \le \tau_K d_K(\hat x,\hat y),
\forall \hat x, \hat y \in \hat \Lambda$.
In this case, $E^s_x$ can be interpreted as the image of a linear map from
$F^s_x$ to $F^u_{\hat x}(K)$, for any prehistory $\hat x $ of $x \in \Lambda$.
Consider therefore $\mathcal{L}_{\hat x}(K):= L(F^s_x, F^u_{\hat x}(K))$
be the space of linear maps from $F^s_x$ to $F^u_{\hat x}(K)$.
$\mathcal{L}(K)$ will denote the vector bundle over $\hat \Lambda$ given by
$\mathcal{L}_{\hat x}(K), \hat x \in \hat \Lambda$, where we consider the
metric $d_K$ on $\hat \Lambda$.
The space $X$ of the H\"{o}lder Section Theorem will be $\hat \Lambda$
endowed with $d_K$ and the
homeomorphism $h$ from the statement of the same theorem is the map
$\hat f ^{-1}: \hat \Lambda \to \hat \Lambda$.
We will also consider the bundle map $\Psi : \mathcal{L}(K) \to \mathcal{L}(K)$
induced by the graph transform associated to the derivative $D f^{-1}(\hat x):
F^s_x \oplus F^u_{\hat x}(K) \to F^s_{x_{-1}} \oplus F^u_{\hat f^{-1}\hat x}(K)$, where $\hat x = (x, x_{-1}, ....) \in \hat \Lambda$.  The mapping
$Df^{-1}(\hat x)$ represents the derivative at $x$ of the local branch
of $f^{-1}$ which takes $x$ into $x_{-1}$, in case $\hat x = (x, x_{-1},...)$
is an arbitrary point of $\hat \Lambda$; this derivative does exist
because we assumed that the critical set of $f$ does not intersect $\Lambda$.
In the sequel we shall use also the notation $Df_s^{-1}(\hat x)$ as being
the inverse of the isomorphism $Df_s(x_{-1}): E^s_{x_{-1}} \to
E^s_{x}$; similarly for the notation $Df_u^{-1}(\hat x)$.
The notion of graph transform used above is explained in \cite{S}.
If we assume that $Df^{-1}(\hat x) = \left( \begin{array}{clr}
                                               A_{\hat x} & B_{\hat x}(K) \\
                                               C_{\hat x}(K) & G_{\hat x}(K)
                                              \end{array} \right)$,
then $A_{\hat x}:F^s_x \to F^s_{x_{-1}}, B_{\hat x}(K): F^u_{\hat x}(K) \to
F^s_{x_{-1}}, C_{\hat x}(K): F^s_x \to F^u_{\hat f^{-1}\hat x}(K),
G_{\hat x}(K) : F^u_{\hat x}(K) \to
F^u_{\hat f^{-1}\hat x}(K)$; let us notice that from the decomposition
above, $B_{\hat x}(K), C_{\hat x}(K)$ and $G_{\hat x}(K)$ depend on $K$, but
$A_{\hat x}$ does not, since the bundle $F^s$ is independent of $K$.
From the definition of graph transform,
\begin{equation}\label{Psi-def}
\Psi_{\hat x}(g) = (C_{\hat x}(K) + G_{\hat x}(K) g)\circ (A_{\hat x}+
B_{\hat x}(K)g)^{-1},
\end{equation}
for any linear map $g \in \mathcal{L}_{\hat x}(K)$.
So it can be noticed that $\Psi_{\hat x}(g) \in
\mathcal{L}_{\hat f^{-1} \hat x}(K)$, for any $\hat x \in \hat \Lambda$.
From construction, $A_{\hat x}$ and $G_{\hat x}(K)$ approximate
$Df_s^{-1}(\hat x)$, respectively
$Df_u^{-1}(\hat x)$, while $|B_{\hat x}(K)| < a_1(\vp), |C_{\hat x}(K)| <
a_1(\vp)$, where $a_1(\cdot)$ is a positive continuous function with
$a_1(0) = 0$.
Hence, if $\vp$ is small enough, then the Lipschitz constant of
$\Psi_{\hat x}$ is
smaller or equal than $\lambda_{\hat x}(K)$, where:
\begin{equation}\label{lambda}
\lambda_{\hat x}(K) := |D f_u^{-1} (\hat x)|\cdot
|D f_s (x_{-1})| + a_2(\vp)
= \frac{|Df_s(x_{-1})|}{|Df_u (x_{-1})|} + a_2 (\vp) < 1,
\end{equation}
and where $a_2(\vp)$ is a positive continuous function in $\vp$,
with $a_2(0) = 0$.
Let us recall now that the metric on $\hat \Lambda$ is $d_K$ which
depends on the constant $K >1$.
In the same spirit as in \cite{S}, we can also assume that the bundle $E:=
\mathcal{L}(K)$ is trivial, otherwise we can replace it with
$E \oplus E'$, for
some complementary bundle $E'$. This replacement does not depend on the
metric $d_K$, since the metric on $E$ is already induced by the product of
the metric
$d_K$ on $\hat \Lambda$ and the usual euclidian metric on the spaces of
linear maps.
We will estimate the local Lipschitz constant $\mu_{\hat x}(K)$ of
$h$ at $\hat x \in \hat \Lambda$, where $h = \hat f^{-1}$ is our base
homeomorphism. Thus, as in the statement of the Pointwise H\"{o}lder
Section Theorem, let $\mu_{\hat x}(K) := \inf\{\frac{d_K(h \hat x,
h \hat y)}{d_K(\hat x, \hat y)}, \hat x \ne \hat y, \hat x, \hat y \in
\hat \Lambda \ \text{and} \ d_K(\hat x, \hat y) < \eta\}$ for some small
$\eta >0$.
Denote also by $\mu(K) := \inf\limits_{\hat x \in \hat \Lambda}
\mu_{\hat x}(K)$.
Then we have:
\begin{equation}
d_K( \hat x, \hat y) = d(x, y)+\frac{d(x_{-1},y_{-1})}{K}+
\frac{d(x_{-2},y_{-2})}{K^2}+.. = d(x, y) + \frac{1}{K} d(\hat f^{-1} \hat x, \hat f^{-1} \hat y)
\end{equation}
Let us denote by $\vp_0$ a positive constant depending only on $f$ such
that $f$ is injective on balls of radius $\vp_0 (\inf\limits_{\Lambda}
|Df_s|)^{-1}$ centered 
on $\Lambda$ and
such that we can apply the Mean Value Inequality  on balls of radius
$\vp_0 (\inf\limits_{\Lambda}|Df_s|)^{-1}$ .
Suppose that $0< \eta < \vp_0$.
If $d_K(\hat x, \hat y) < \eta$, and $d_K(\hat f^{-1}\hat x,
\hat f^{-1}\hat y) > \eta$, then $d_K(\hat x, \hat y) <
(|D f_u(x_{-1})|+\frac{1}{K})
d_K(\hat f^{-1}\hat x, \hat f^{-1}\hat y)$ since $|D f_u(x_{-1})|+\frac{1}{K}
> 1$.
So, with the assumption that $d_K(\hat x, \hat y) <  \eta$, let us suppose
also that $d_K(\hat f^{-1}\hat x, \hat f^{-1}\hat y) < \eta$.
Hence $d(x_{-1}, y_{-1}) < \eta$ and, from our assumption
it follows also that $d( x, y) < \eta$, so, using the Mean 
Value Inequality, we obtain that:
\begin{equation}\label{small}
d_K(\hat x, \hat y) \le (|Df_u (x_{-1}')| + \frac{1}{K})
d_K(\hat f^{-1} \hat x, \hat f^{-1} \hat y) = (|Df_u(x_{-1}')|+ \frac{1}{K})
d_K(h \hat x, h \hat y),
\end{equation}
where $x_{-1}'$ is some point with $d( x_{-1}, x_{-1}') < \eta$.
This implies that the constant $\mu_x$ which appears in the Pointwise
H\"{o}lder Section Theorem is represented in our situation by
$\mu_{\hat x}(K)$ and, as we saw in (~\ref{small}),
\begin{equation}\label{mu}
\mu_{\hat x}(K) \ge
(|Df_u (x_{-1})| + \frac{1}{K} + \omega(|Df_u|, \eta) )^{-1},
\end{equation}
where $\omega(|Df_u|, \eta)$ is the maximum oscillation of $|Df_u|$ on a ball
of radius $\eta$ centered at an arbitrary point of $\Lambda$, and we used 
above that $|Df_u(x_{-1}')| \le |Df_u(x_{-1})| + \omega(|Df_u|, \eta)$.

Next, we show that $\Psi_{\hat x}$ is Lipschitz in $\hat x$; recall that we
assumed that $\mathcal{L}(K)$ is a trivial bundle, so we can identify
all the 1-dimensional complex spaces $\mathcal{L}_{\hat x}(K)$ with $\bb C$,
and do this independently of $K$.
We wish to prove that there exists a constant $\Theta_K > 0$ such that
\begin{equation}\label{psi}
|\Psi_{\hat x}(g) - \Psi_{\hat y}(g)| \le \Theta_K d_K(\hat x, \hat y),
\forall \hat x, \hat y \in \hat \Lambda, \forall g \in \bb C, |g| \le 1
\end{equation}
From the fact that $f$ is smooth and $F^s$ depends Lipschitz in $x \in
\Lambda$, while $F^u_{\hat x}(K)$ depends Lipschitz in $\hat x \in \hat
\Lambda$, it follows that $A_{\hat x}$ depends Lipschitz in $x$ (with respect 
to the euclidian metric induced on $\Lambda$) and
$B_{\hat x}(K), C_{\hat x}(K), G_{\hat x}(K)$ depend Lipschitz in $\hat x$ 
(with respect to the metric $d_K$).
Recall from (~\ref{Psi-def}) that $\Psi_{\hat x}(g) = 
(C_{\hat x}(K) + G_{\hat x}(K) g)
\cdot (A_{\hat x}+
B_{\hat x}(K)g)^{-1}$, for any linear map $g \in \mathcal{L}_{\hat x}(K)$.
But in our case, $g, A_{\hat x}, B_{\hat x}(K), C_{\hat x}(K), G_{\hat x}(K)$
are just complex numbers. It is enough to show that $\hat x \to (A_{\hat x}
+B_{\hat x}(K)g)^{-1}$ is Lipschitz.
But since we work with complex numbers we have $|(A_{\hat x}
+B_{\hat x}(K)g)^{-1}-(A_{\hat y}
+B_{\hat y}(K)g)^{-1}| = \left| \frac{(A_{\hat y}-A_{\hat x}) + (B_{\hat y}(K)
- B_{\hat x}(K)) g}{(A_{\hat x}+B_{\hat x}(K)g)(A_{\hat y}+B_{\hat y}(K)g)} 
\right|$. Now we use the fact that $A_{\hat x}, B_{\hat x}(K)$ depend 
Lipschitz in $\hat x$ and $|B_{\hat x}(K)| < a_1(\vp) << 1, \forall \hat x 
\in \hat \Lambda$.
Thus, for $|g| \le 1$ we get that $|A_{\hat x}+B_{\hat x}(K)g|$ is 
uniformly (in $\hat x$) bounded away 
from 0, since $|A_{\hat x}|$ approximates $|Df_s^{-1}(\hat x)|$ (and we know
that $|Df_s^{-1}(\hat x)| \ge (\sup\limits_{\Lambda}|Df_s|)^{-1} > 0$),
and $|B_{\hat x}(K)|$ is very small in comparison to $|A_{\hat x}|$.
In conclusion we obtained the Lipschitz continuity of $\Psi$, hence 
inequality (~\ref{psi}). 

Let us check now the condition 3) of the Pointwise H\"{o}lder Section
Theorem with $\alpha = 1$.
Using the relations in (~\ref{lambda}) and (~\ref{mu}), we have that:
\begin{equation}\label{rho}
\aligned
\rho(1, K)&:= \sup\limits_{\hat x \in \hat \Lambda} \lambda_{\hat x} \cdot
\mu_{\hat x}(K)^{-1} \le  (\frac{|Df_s(x_{-1})|}{|Df_u (x_{-1})|} + a_2 (\vp))
\cdot (|Df_u (x_{-1})| + \frac{1}{K} + \omega(|Df_u|, \eta) ) =\\
&= (\frac{|Df_s(x_{-1})|}{|Df_u(x_{-1})|} + a_2(\vp)) \cdot (\frac{1}{K} +
\omega(|Df_u|, \eta)) + \frac{|Df_s(x_{-1})|}{|Df_u (x_{-1})|} \cdot
|Df_u(x_{-1})| + a_2(\vp) |Df_u(x_{-1})| \le \\
&\le |Df_s(x_{-1})| + M(\vp, \eta, K) < 1,
\endaligned
\end{equation}
where $M(\vp, \eta, K)$ is a positive continuous function in $\vp$, $\eta$,
and $K$ with $M(0, 0, \infty) = 0$. This is why in the last inequality
of (~\ref{rho}) we were able 
to take $M(\vp, \eta, K) <
1 - \sup\limits_{\Lambda}|Df_s|$,
for $\vp$ and $\eta$ small enough and $K$ large enough. The values of such
$\vp, \eta, K$ depend only on $f$.
Therefore, we found that in this case condition 3) of the Pointwise Section
Theorem is satisfied for $\alpha = 1$.

Now, according to (~\ref{psi}), it follows that condition 2) from the
statement of the Pointwise Section
Theorem is satisfied as well, so all the conditions of the Pointwise
H\"{o}lder Section Theorem hold and we get that the unique invariant section
$\sigma$ is Lipschitz.
But in our case this unique invariant section $\sigma$ is just the stable
bundle, $\sigma(\hat x) = E^s_x, \forall \hat x \in \hat \Lambda$,
hence there exists a constant $C_K$ depending on $K$ such that:
\begin{equation}\label{stablelip}
d(E^s_x, E^s_y) \le C_K d_K(\hat x, \hat y), \forall \hat x,
\hat y \in \hat \Lambda
\end{equation}
Let us denote now by $\lambda_s := \inf\limits_{z \in \Lambda}
|Df_s(z)|$, and take $\tilde \vp_0:= \lambda_s \vp_0$, where the number 
$\vp_0$ has been introduced earlier; clearly $\tilde\vp_0 \ne 0$
since the critical set of $f$ avoids $\Lambda$.
We want to prove that (~\ref{stablelip}) implies that, in fact, $x \to
E^s_x$ is Lipschitz.

Case 1:
Let us then assume first that $x, y \in \Lambda$ with $d(x, y) \ge \tilde
\vp_0$.
If $\Delta_0$ denotes the diameter of $\Lambda$, then
\begin{equation}\label{case1}
\aligned
d_K(\hat x, \hat y) &\le d(x, y) + \frac{2\Delta_0}{K} \le
d(x, y) + \frac{2\Delta_0}{K} \cdot \frac{d(x, y)}{\tilde\vp_0} \le\\
& \le  d(x, y)(1+\frac{2\Delta_0}{K \tilde\vp_0}) < 
d(x, y)(1+\frac{2\Delta_0}{\tilde\vp_0})
\le C' d(x, y),
\endaligned
\end{equation}
with $C' >0 $ a constant independent of $K$.

Case 2:

Now suppose that $0 < d(x, y) < \tilde\vp_0$ for some $x, y \in \Lambda$.
We consider here the map $f$ restricted to $\Lambda$.
We will say that $(x, x_{-1} ..., x_{-n})$ are consecutive preimages of $x$
in $\Lambda$ if $f(x_{-1}) = x, f(x_{-2}) = x_{-1}, ..., f(x_{-n}) = x_{-n+1}$
and $x_{-j} \in \Lambda, \forall j = 1..n$.
Consider $n = n(x, y)$ to be the largest
positive integer such that there exist consecutive preimages of $x$ and of
$y$, $(x, x_{-1}^*, ..., x_{-n}^*)$ and $(y, y_{-1}^*, ..., y_{-n}^*)$ with
$d(x_{-i}^*, y_{-i}^*) < \vp_0, i = 1, .., n$.
Since $n$ is the largest such integer, it follows that, for some $x_{-n-1}^*
\in f^{-1}(x_{-n}^*)$ and $y_{-n-1}^* \in f^{-1}(y_{-n}^*)$, with
$d(x_{-n-1}^*, y_{-n-1}^*) < \vp_0 \lambda_s^{-1}$, we have:
\begin{equation}\label{largest}
\vp_0 < d(x_{-n-1}^*, y_{-n-1}^*) \le \lambda_s^{-1} d(x_{-n}^*, y_{-n}^*)
\end{equation}
We also obtain
\begin{equation}\label{mvi}
d(x_{-i}^*, y_{-i}^*) \le \lambda_s^{-i} d(x, y), i=1,.., n
\end{equation}
From (~\ref{largest}) and (~\ref{mvi}), we obtain that $d(x_{-n-1}^*,
y_{-n-1}^*) \le \lambda_s^{-n-1} d(x, y)$. This implies that, for any
complete prehistories
$\hat{x^*}, \hat{y^*}$ of $x, y$, which start with the consecutive preimages
$(x, x_{-1}^*,..., x_{-n}^*)$, $(y, y_{-1}^*, ..., y_{-n}^*)$ considered above,
we have
\begin{equation}\label{dist}
\aligned
d_K(\hat{x^*}, \hat{y^*}) &= d(x, y) + \frac{d(x_{-1}^*, y_{-1}^*)}{K}
+.... \le \\
& \le d(x, y) + \frac{1}{\lambda_s K} d(x, y)+...+ \frac{1}{\lambda_s^n K^n}
d(x, y) + \frac{2 \Delta_0}{K^{n+1}}
\endaligned
\end{equation}
Assume that $K$ is fixed such that $K > \lambda_s^{-2}$ and such that
$M(\vp, \eta, K) < 1-\sup\limits_{\Lambda}|Df_s|$ for some $\vp <1$ and 
some $\eta < \vp_0$. Then , from (~\ref{largest}) and (~\ref{mvi}),
$\vp_0 < \lambda_s^{-n-1} d(x, y) < K^{n+1} d(x, y)$, which implies that
$\frac{1}{K^{n+1}} < \frac{d(x, y)}{\vp_0}$.
Introducing this inequality in (~\ref{dist}), one sees that there exists
a positive constant $C''$ such that for our chosen prehistories
$\hat{x^*}, \hat{y^*}$, of $x$, respectively $y$,
\begin{equation}\label{case2}
d_K(\hat{x^*}, \hat{y^*}) \le C'' d(x, y)
\end{equation}
By considering now both Case 1, (~\ref{case1}), and Case 2, (~\ref{case2}),
together with (~\ref{stablelip}), we obtain the Lipschitz continuity of
the stable spaces with respect to their base points -- i.e there exists
a positive constant $\Upsilon$ such that for all $x, y$ from $\Lambda$,
 $d(E^s_x, E^s_y) \le \Upsilon d(x, y)$. This implies immediately
that also $\phi^s$ is Lipschitz on $\Lambda$.

Now, we will prove the uniform Lipschitz continuity of the stable distribution 
and of $\phi^s$ along the stable leaves in the holomorphic case. We notice that, 
since $\Lambda$ is compact, one can construct
local stable manifolds of uniform size $r$ at all points of $\Lambda$, 
if $r>0$ is small enough.
If $y$ is a point in a manifold $W^s_r(x)$, but $y$ is not necessarily
in $\Lambda$, we shall call stable space at $y$, denoted by $E^s_y$, the 
tangent space at $W^s_r(x)$ at $y$.
We see that the spaces $E^s_y$ vary smoothly when $y$ moves inside $W^s_r(x)$
for $x$ fixed. So the existence of a constant $\Xi$ like in the statement
is conditioned only on the boundedness of the ``curvature''
of these local stable manifolds.  Assume then that there exists a sequence 
$z_n \in \Lambda$ such that the Lipschitz constants $L_n$ of the maps $g_n$
converge to infinity, where $g_n(y) := E^s_y, y \in W^s_r(z_n)$.
Since $\Lambda$ is compact, the sequence $(z_n)_n$ has at least one 
convergent subsequence and without loss of generality we can assume that 
this subsequence is again $(z_n)_n$ and $z_n \to z$. 
If $x$ is an arbitrary 
point in $\Lambda$, then $W^s_r(x)$ is an 
analytic disk which is given as the image of an analytic map $h_x$ from the 
unit disk $\Delta$ to $\bb C^2$. 
We denote by $h_n$ the map $h_{z_n}$, for $n$ positive integer.
But from the hyperbolicity condition, the analytic maps $h_x$ vary continuously in 
$x\in \Lambda$, hence also $h_n$ vary continuously in $n$.
The norm on $\Delta$ of the second derivative of $h_n$ bounds 
the Lipschitz constant $L_n$ of 
the map $g_n$, for all $n$.
Notice however that, since $h_n$ are holomorphic and vary continuously 
in $n$, also the second 
derivatives of the maps $h_n$ vary continuously in $n$. 
Therefore, since we assumed $z_n \to z \in \Lambda$, we obtain that $L_n$ 
are bounded by some finite positive constant $L$. So the map $y \to E^s_y$ 
is $L$-Lipschitz on $W^s_r(x), \forall x \in \Lambda$. Then, due to the 
smoothness of $f$,
there exists a small $r>0$ and $\Xi>0$ such that for any $x \in
\Lambda$ and any points $y, z \in W^s_r(x)$, we have 
$| \phi^s(y) - \phi^s(z) |
\le \Xi \cdot d(y,z)$. 

\end{proof}

\begin{pro}\label{proC}
Let $f:\bb P^2 \to \bb P^2$  holomorphic, with Axiom A and such that
$\mathcal{C}_f \cap \Lambda = \emptyset$ for a basic set $\Lambda$ of unstable 
index 1. Let also a prehistory $C$ of a point $x$ in $\Lambda$, with 
respect to $f$.
If $m := n(C)$, $C=( x, x_{-1}, ...,x_{-m})$ and 
$y$ is an arbitrary point in $\Lambda(C, \vp)$, with the corresponding
prehistory $(y, y_{-1},...,y_{-m})$ $\vp$-shadowed by $C$, then we have:
$\frac{1}{C_1} \le \frac{|Df_s^m(y_{-m})|}{|Df_s^m(x_{-m})|} < C_1,$
where $C_1 > 1$ is a constant independent of $m$ and $C$.
\end{pro}

\begin{proof}
From the fact that $(y, ..., y_{-m})$ is an $m$-prehistory of $y$ in 
$\Lambda$ we know in particular that $y_{-m} \in \Lambda$, hence 
there exists a local stable manifold 
through $y_{-m}$ of size $\vp$. Let us take also $\hat {x}$ be any complete 
prehistory in $\Lambda$ of 
$x$, starting with $(x,x_{-1},...,x_{-m})$. Set $\hat x_{-m}:=
\hat f^{-m}(\hat x)$. In this case
$W^u_{\vp}(\hat x_{-m})$ intersects $W^s_{\vp}(y_{-m})$ in a unique point $z$.
It follows from the local product structure of $\Lambda$ that $z$ belongs
to $\Lambda$. 
From the fact that $y$ belongs to $\Lambda(C,\vp)$ and $(y,..., y_{-m})$ is
its prehistory $\vp$-shadowed by $C$, we know that 
$d(f^ix_{-m}, f^iy_{-m})<\vp$ for all $i=0,1,...,m$. 
Also from the fact that $z\in W^s_\vp(y_{-m})$ it follows that $d(f^iz,
f^iy_{-m})<\vp$ for all $i=0,1,...,m$. From the last two inequalities we get 
that 
$d(f^i x_{-m}, f^i z) < 2\vp$ for all $i=0,1,...,m$. But, since 
$z \in W^u_\vp(\hat x_{-m}) \cap W^s_\vp(y_{-m})$, we have that there 
exist constants $\tilde c >0$ and 
$\gamma \in (0,1)$ such that for all $i=0,1,...,m$, 
\begin{equation}\label{laminated}
d(f^i x_{-m}, f^i z) < \tilde c \gamma^{m-i} \ \text{and}\ 
d(f^i y_{-m}, f^i z) < \tilde c \gamma^{i}.
\end{equation}
Now from Theorem~\ref{thm0}, $\phi^s(y)$ depends 
Lipschitz continuously on $y \in \Lambda$. 
This, together with (\ref{laminated}), implies that there exists a 
constant $K' > 0$ such that: 
$\left|\sum_{j = 0}^m \phi^s(y_{-j}) - \sum_{j=0}^m \phi^s(x_{-j})\right| \le 
 \left|\sum_{j = 0}^m \phi^s(y_{-j})-\sum_{j = 0}^m \phi^s(f^{m-j}z)\right| +
\left|\sum_{j = 0}^m \phi^s(f^{m-j}z)-\sum_{j = 0}^m \phi^s(x_{-j})\right| 
\le K'( \sum_{j=0}^m d(y_{-j}, f^{m-j}z) + \sum_{j=0}^m d(f^{m-j}z, 
x_{-j})) \le 2K'\tilde c \cdot \sum_{j=0}^m \gamma^j < K''$, where $K''$ is a constant independent of $m$ and $\vp$.
Hence the statement of the proposition follows immediately from the previous 
inequalities.
\end{proof}

\begin{pro}\label{pro4}
Let $f:\bb P^2 \to \bb P^2$  holomorphic, with Axiom A and such that
$\mathcal{C}_f \cap \Lambda = \emptyset$ for a basic set $\Lambda$ of unstable 
index 1. Denote $\chi_u := \mathop{\sup}\limits_\Lambda |Df_u|$.

(a) Then we have that $t^s_n(\vp) \ge t^s_{np}(\vp)$ and that $t^s = 
t^s_n$, for any positive integers $n, p$ and any $\vp >0$. 

(b) For $\vp < \vp_0$, and
$\rho$ an arbitrary number in the interval $(0, \chi_u^{-1})$, 
denote by $\rho_n:= \vp \cdot \rho^n, n >1$. Then $P^-_{f^n}(t\phi^s_n, \rho_n) = P^-_{f^n}(t\phi^s_n)$, for
any $t$; consequently $t^s_n(\rho_n) = t^s_n = t^s, n >1$.

\end{pro}

\begin{proof}

(a) First we make the following notations.
If $m$ is a positive integer, denote by 
$\mathcal{C}^n_m:= \{(y, y^n_{-1},..., y^n_{-m}) \in \Lambda^{m+1}, \
\text{such that} \ f^n(y^n_{-i}) = y^n_{-i+1}, i=1,..,m, \text{and} \ y_0=y \}$.
Let also $\mathcal{C}^n_*:= \mathop{\cup}\limits_{m\ge 0} \mathcal{C}^n_m$
be the set of prehistories of finite length for $f^n$ in $\Lambda$.
Now, if $n$, $p$ and $\vp>0$ are fixed, we consider an arbitrary number 
$t \in ( t^s_n(\vp), t^s_n(\vp)+1)$.
From the definition of $t^s_n(\vp)$, we get that, for $N$ large, there 
exists an $\vp$-covering $\Gamma$ of $\Lambda$, $\Gamma
\subset \mathcal{C}^n_*$ with $n(C) \ge N, \forall C \in \Gamma$ and:
\begin{equation}\label{eq1}
\sum_{C\in \Gamma} exp(S^-_{n(C)}(t \phi^s_n(C))) < 
exp(-(t^s_n(\vp)+1) n(2p-1)\sup\limits_{\Lambda}|\phi^s|)
\end{equation}
For every $C \in \Gamma$, let us divide $n(C)$ by $p$, and obtain 
$n(C) = p\cdot m(C)+k(C)$, where $0 \le k(C) < p$.
If $C = (y, y^n_{-1}, ..., y^n_{-n(C)})$, then denote by $C'$ the 
$m(C)$-prehistory of $y$ with respect to $f^{np}$ given by $C' = 
(y, z^{np}_{-1}, ..., z^{np}_{-m(C)})$, where $z^{np}_{-1} := 
y^n_{-p}, ..., z^{np}_{-m(C)}:= y^n_{-p m(C)}$.
Then it is easy to see that $\Lambda(C, \vp) \subset \Lambda(C', \vp)$, 
for all $C \in \Gamma$. 
Denote by $\Gamma'$ the collection of all the prehistories $C'$ associated 
by the above procedure to the prehistories $C$ from $\Gamma$.
We calculate now the consecutive sum 
$
S^-_{n(C)} \phi^s_n (C) 
=\phi^s_n(y) + ...+ \phi^s_n(y^n_{-m(C)p})+\phi^s_n(y^n_{-m(C)p-1})+...+
\phi^s_n(y^n_{-n(C)}) 
= \log |Df_s^{n(pm(C)+1)}(y^n_{-m(C)p})| + \log |Df_s^{nk(C)}(y^n_{-n(C)})|.$
On the other hand
$$
\aligned
S^-_{m(C)}\phi^s_{np}(C') 
&= \phi^s_{np}(y)+...+\phi^s_{np}(z^{np}_{-m(C)}) \\
&= \phi^s(y^n_{-m(C)p}) + \phi^s(f y^n_{-m(C)p}) +...+\phi^s(y)+\phi^s(fy)+...+
\phi^s(f^{np-1} y) \\
&= \log |Df_s^{np(m(C)+1)}(y^n_{-m(C)p})|.
\endaligned
$$
These last two relations show that 
$S^-_{n(C)}\phi^s_n (C) 
= S^-_{m(C)}
\phi^s_{np}(C') + \log |Df_s^n(y)| + \log |Df_s^{nk(C)}(y^n_{-n(C)})| - 
\log|Df_s^{np}(y)|.$
Using that $k(C) < p$ and the last equality, we obtain that
$$
|S^-_{n(C)}\phi^s_n (C) - S^-_{m(C)} \phi^s_{np}(C') | \le  
 n (p-1) \cdot \sup\limits_{\Lambda} |\phi^s| + 
|\log |Df_s^{nk(C)}(y^n_{-n(C)})|| \le n(2p-1)\cdot \sup\limits_{\Lambda}
|\phi^s|
$$
Therefore 
\begin{equation}
\aligned
\inf &\{ \mathop{\sum}\limits_{C'\in \Gamma'}  exp(S^-_{m(C)}(t \phi^s_{np}
(C'))), \Gamma' \subset \mathcal{C}^{np}_* \vp - \text{covers} \ 
\Lambda \} \le \\ 
& \le [\mathop{\sum}\limits_{C\in \Gamma} 
exp(S^-_{n(C)}(t \phi^s_n(C)))] \cdot 
exp(tn(2p-1)\sup\limits_{\Lambda} |\phi^s| ) < 1.
\endaligned
\end{equation}
The last inequality follows since $t <  t^s_n(\vp)+1$ and from the 
way we chose $\Gamma$ in the begining of the proof.
But from the definition of $P^-_{np}$, we obtain then that 
$t \ge t^s_{np}(\vp)$.
However since $t$ was taken arbitrarily in the finite interval 
$(t^s_n(\vp), t^s_n(\vp)+1)$, it follows that $t^s_n(\vp) \ge t^s_{np}(\vp)$.
The inequality $t^s(\vp) \ge t^s_n(\vp)$ implies that $t^s \ge t^s_n$, $n \ge 1$. We want to prove now the opposite inequality, i.e $t^s \le t^s_n$
(actually the same proof shows more generally, that  
$P^-_{f^n}(t\phi^s_n) = n P^-_f(t\phi^s)$).
Indeed, let us consider an arbitrary $t > t^s_n$, for a fixed integer $n$.
For a given $\vp>0$, let $\bar\vp_n>0$ satisfying the following conditions: for any $y, z$ with 
$d(y,z) < \bar\vp_n$ we have $d(f^jy, f^jz) < \vp, 0 \le j \le n$, and also
 $P^-_{f^n}(t\phi^s_n, \bar\vp_n) < 0$. Hence for all $m$ large, there exists an $(m,\bar\vp_n)$-cover 
$\Gamma^n_m$ of $\Lambda$ (i.e $\Gamma^n_m$ is a collection of $m$-prehistories $C'$ with respect to $f^n$, so 
that 
$\Lambda = \mathop{\cup}\limits_{C' \in \Gamma^n_m}\Lambda(C', \bar\vp_n)$),
satisfying: \ $\mathop{\sum}\limits_{C' \in \Gamma^n_m} e^{S^-_m(t\phi^s_n)(C')} < 1$. Now, out of every $C'$ we will form a prehistory $C$ with respect to $f$ in the canonical way, i.e if $C' = (y, y_{-n}, ..., y_{-nm})$, then $C=(y, f^{n-1}y_{-n}, ..., y_{-n}, ..., f(y_{-nm}), y_{-nm})$.
Also, from the condition satisfied by $\bar\vp_n$, we see that $\Lambda(C', \bar\vp_n) \subset \Lambda(C, \vp)$; so, if $\Gamma_{nm}$ denotes the collection of prehistories $C$ of length $nm$ (with respect to $f$) obtained as above from the prehistories $C'$ of $\Gamma^n_m$, we obtain that $\Gamma_{nm}$ is an $(nm, \vp)$ cover of $\Lambda$.
Moreover, as found above, $S^-_{nm}(t\phi^s)(C) = S^-_m(t\phi^s_n)(C') + \log|Df_s(y)| - \log|Df_s^n(y)|$.
These facts imply  
that $\mathop{\sum}\limits_{C \in \Gamma_{nm}} e^{S^-_{nm}(t\phi^s)(C)} < M_n$, where $M_n$ is a constant depending only on $n$. Therefore if we let $m \to \infty$ (and keep $n$ fixed), we see that 
$P^-_f(t\phi^s, \vp) \le 0 \Rightarrow t \ge t^s(\vp)$. 
But $0<\vp<\vp_0$ was arbitrary and $t$ was taken arbitrarily larger than $t^s_n$, hence
  $t^s_n \ge t^s$.
This proves the equality $t^s = t^s_n, n \ge 1$.

(b)  First from the proof of Proposition \ref{proC} we know that for all $m \ge 1$, and prehistory $(x, x_{-1}, ..., x_{-m})$ of $x$ in $\Lambda$, $\frac{1}{C_1(\vp)} \le \frac{|Df_s^m(y_{-m})|}{|Df_s^m(x_{-m})|} \le C_1(\vp)$, for $(y, y_{-1}, ..., y_{-m})$ an $m$-prehistory of $y$, $\vp$-shadowed by $(x, x_{-1}, ..., x_{-m})$.
The proof of Proposition \ref{proC} implies also that $C_1(\vp) \le C_2 \cdot \vp, 0 < \vp < \vp_0$, for some constant $C_2>0$.
Let us consider now the situation for $f^n$ for some fixed $n \ge 1$.
Consider $(x, x_{-n}, ..., x_{-np})$ a $p$-prehistory of $x$ in $\Lambda$  (with respect to $f^n$), and let 
$(y, y_{-n}, ..., y_{-np})$ be another $p$-prehistory in $\Lambda$ which is $\rho_n$-shadowed by 
$(x, x_{-n}, ..., x_{-np})$.
Then, if $d(y_{-np}, x_{-np}) < \rho_n < \vp \rho^n$, we get that $d(f^j(y_{-np}), f^j(x_{-np})) < \vp, 0 \le j\le n$, and similarly we obtain that $d(f^j(y_{-np}), f^j(x_{-np})) < \vp, 0 \le j\le np$.
Therefore the $np$-prehistory with respect to $f$, $(y, y_{-1}, ...., y_{-np})$ is $\vp$-shadowed by $(x, x_{-1}, ..., x_{-np})$.
So we can apply Proposition \ref{proC} in this case to obtain similar inequalities for prehistories of $f^n$:
\begin{equation}\label{n-iterate}
\frac{1}{C_1(\vp)} \le \frac{|Df_s^{np}(y_{-np})|}{|Df_s^{np}(x_{-np})|} \le C_1(\vp),
\end{equation}
for any $p \ge 1$.
Next, take $C$ an arbitrary $p$-prehistory in $\Lambda$, with respect to $f^n$, for $n$ fixed.
If $\vp'$ is an arbitrary number in the interval $(0, \rho_n)$, we see that the set $\Lambda(C, \rho_n)$ can be covered with at most $(\frac{\rho_n C_1(\vp)}{\vp'})^4$ sets of the form $\Lambda(C', \vp')$, where $C'$ are $p$-prehistories with respect to $f^n$.
Thus, recalling the definition of $P^-_{f^n}(t\phi^s_n, \rho_n), P^-_{f^n}(t\phi^s_n, \vp')$ and 
inequality (\ref{n-iterate}), we conclude that: 
$P^-_{f^n}(t\phi^s_n, \rho_n) = P^-_{f^n}(t\phi^s_n, \vp') = P^-_{f^n}(t\phi^s_n)$.
The last equality above follows from the fact that $P^-_{f^n}(t\phi^s_n, \vp') \to 
P^-_{f^n}(t\phi^s_n)$ when $\vp' \to 0$.
Hence, recalling also the conclusion of part (a), we get $t^s_n(\rho_n) = t^s_n = t^s, n >1$.
\end{proof}

\section{Estimates from above and below for the stable dimension in the general holomorphic case using the inverse pressure of iterates}

Given a map $f$ and a basic set $\Lambda$ as in Proposition \ref{proC}, define 
$\lambda_s:= \inf\limits_{\omega 
\in \Lambda}|Df_s(\omega)|$ and $\chi_s := \sup\limits_{\omega \in \Lambda}
|Df_s(\omega)|$. Remark that $\lambda_s > 0$ since we assumed that $\Lambda \cap \mathcal{C}_f = \emptyset$. For every positive integer $n$ and small positive number $\vp$, let 
$t^s_n(\vp)$ (respectively $t^s_n$) be the unique zero of the function 
$t \to P^-_{f^n}(t\phi^s_n, \vp)$ (respectively $t \to P^-_{f^n}(t\phi^s_n)$), where $\phi^s_n(y):=
\log |Df^n_s(y)|, y \in \Lambda$.
 
\begin{thm}\label{thm1}
Let $f:\bb P^2 \to \bb P^2$ be a holomorphic non-degenerate map with Axiom A and 
$\Lambda$ a basic set of $f$ with unstable index 1. Assume also that the 
critical set of $f$, $\mathcal{C}_f$ does not intersect $\Lambda$. 

(a) Then for every $x\in \Lambda$, we have 
$\delta^s(x) \le t^s_n(\rho_n) = t^s$, where $\rho_n >0$ are small numbers 
of the form $\vp\rho^n, n \ge 1$, where $\chi_u := \sup\limits_\Lambda |Df_u|$, $\rho>0$ is an arbitrary number smaller than $\chi_u^{-1}$, and $\vp < \text{min}\{\vp_0, r_0\}$.

(b) For all positive numbers $\vp< \vp_0$, and $\eta >0$, we get $\delta^s(x) + \eta \ge
t^s_{n}(\vp)$, where $ n \ge n(\vp, \eta)$ and $n(\vp, \eta)$ is a positive integer satisfying 
$n(\vp, \eta) > \frac{4\log \frac {1}{\vp}}{\eta \cdot \log \chi_s^{-1}}$.
In particular, if $\eta = \vp$ small enough, we get $t + \vp \ge t^s_n(\vp)$, for $n \ge (\frac{1}{\vp})^{1.1}$.

\end{thm}

\begin{proof}
(a) According to Proposition \ref{pro4}, we have $t^s_n(\rho_n) = t^s$.
From the Theorem of \cite{MU2}, recalled also in the Introduction, we have that 
$\delta^s(x) \le t^s$. Hence $\delta^s(x) \le t^s_n(\rho_n), n >1$.

(b) We prove now the inequality $\delta^s(x) + \eta \ge 
 t^s_{n}(\vp)$ for $\vp>0$ small enough (to be determined next),
$\eta>0$ small, and $ n \ge n(\vp, \eta)$.

First let us notice that, from definition, $\delta^s(x) \le 2$.
Let us take an arbitrary $t$ with $\delta^s(x) < t < 3$. Recall also that $\vp_0$ has been introduced earlier as 
a positive constant so that we can apply the Mean Value Inequality 
for $f$ on balls of radius $\vp_0 (\inf\limits_{\Lambda}|Df_s|)^{-1}$, and
also such that $f$ is injective on balls of radius 
$\vp_0 (\inf\limits_{\Lambda}|Df_s|)^{-1}$ centered on $\Lambda$.

Consider now $N_0(\vp)$ to be the smallest cardinality of a covering of $\Lambda$ with balls of radius $\vp$.
Then if $\beta= \ovl{\text{dim}}_B(\Lambda)$ denotes the upper box dimension of $\Lambda$, and $\beta_0 < \beta < \beta_1$, we will have that $(\frac{1}{\vp})^{\beta_0} < N_0(\vp) < (\frac{1}{\vp})^{\beta_1}$, for $\vp>0$ small enough. 
With $\vp>0$ and $\eta>0$ fixed, consider $n(\vp, \eta)$ be the smallest positive integer $n$ such that 

\begin{equation}\label{eta}
N_0(\vp) \cdot \chi_s^{n\eta} < 1
\end{equation}

This is satisfied if, for example, $n(\vp, \eta) > \frac{4 \log \frac{1}{\vp}}{\eta\cdot \log \frac{1}{\chi_s}}$.
In the sequel we consider $\vp$ with 
$0< \vp < \text{min}\{ \vp_1/2, r, d(\Lambda, \mathcal{C}_f)/4 \}$.
We shall prove that, for such an $\vp$ and $\eta>0$, the inequality
$t + \eta > t^s_n(\vp)$ holds for $n \ge n(\vp, \eta)$. 

Define now a constant $0 < \tilde\alpha < 1$ which depends only on $f$ and on
$\Lambda$, such that for all $x' \in \Lambda$ and $0 < r' << 
\text{diam}\Lambda $ , we have that $W^s_{r'}(y')$ intersects $W^u_{r'}(\hat
z')$ for all points $y', z' \in B(x', \tilde\alpha r')$ and all 
prehistories $\hat z' \in \hat \Lambda$ of $z'$. The existence of such
a constant follows from the 
transversality of stable and unstable manifolds.  

Next let us cover the compact set $\Lambda$ with a finite number of balls 
$B(y_1, \tilde\alpha\vp/4), ..., B(y_s, \tilde\alpha\vp/4)$ which are 
centered at points of 
$\Lambda$. Let us choose one such ball and denote its intersection 
with $\Lambda$ by $Y$.

We will show now that there exists a positive integer $m$ such that all
local unstable manifolds $W^u_{\vp}(\hat y)$ intersect the set $f^{-m}(W)$,
for all prehistories $\hat y \in \hat \Lambda$ of all points $y \in Y$, where we recall that $W:= W^s_r(x) \cap \Lambda$.

Indeed, from the transitivity of $f$ on $\Lambda$, there exists a positive 
integer $m$ and a point $z \in Y \cap \Lambda$ such that
$f^m(z) \in B(x, \tilde\alpha\vp/2) \cap \Lambda$.
Take now a complete prehistory $\hat y \in \hat \Lambda$ of an arbitrary
point $y$ from $Y$.
From the fact that $Y$ is contained in a ball of radius $\tilde\alpha\vp/4$,
we can conclude that $W^s_{\vp/2}(z)\cap W^u_{\vp/2}(\hat y) \ne 
\emptyset$ and 
denote this intersection (which is a point) by $\xi$. From the local product 
structure $\xi$ belongs to $\Lambda$.
We have also that $f^m(\xi) \in W^s_\vp(f^m z) \cap \Lambda$.
Take now $\widehat{f^m \xi}$ to be the prehistory in $\Lambda$ of $f^m\xi$ 
given by $(f^m \xi,
f^{m-1}\xi,...,\xi, \xi_{-1}, ...)$, where $\hat \xi:= (\xi, \xi_{-1},...)$ 
is the
prehistory of $\xi$ $\vp/2$-shadowed by $\hat y$; such a prehistory of $\xi$
exists since $\xi \in W^u_{\vp/2}(\hat y)$.
So, we get that there exists a local unstable manifold 
$W^u_{\vp/2}(\widehat{f^m\xi})$ which intersects $W^s_{\vp/2}(x)$ in a 
point $\zeta$; again from
the local product structure, $\zeta \in \Lambda$ and since $\zeta \in 
W^s_{\vp/2}(x)$, we obtain that $\zeta \in W$.
If we consider $\zeta_{-m}$ the $m$-th preimage of $\zeta$ obtained from
the fact that $\zeta \in W^u_{\vp/2}(\widehat{f^m\xi})$, we will have 
$d(\zeta_{-m}, \xi) < \vp/2$.
Combining with the fact that $\hat \xi$ corresponds to a prehistory of
$\xi$ $\vp/2$-shadowed by $\hat y$, it follows that $\zeta_{-m} \in 
W^u_\vp(\hat y) \cap f^{-m}W$. We may denote the point $\zeta_{-m}$ also
by $\zeta_{-m}(\hat y)$ when we want to emphasize its dependence on $\hat y$.

Therefore, we proved that the set $f^{-m}W$ intersects all unstable manifolds
$W^u_\vp(\hat y)$ for all prehistories $\hat y \in \hat \Lambda$ of points
$y$ from $Y$.

From the fact that $\zeta \in W^u_{\vp/2}(\widehat{f^m\xi})$, it follows that
$d(\zeta_{-m}, \xi) < \vp/2, d(f\zeta_{-m}, f\xi) < \vp/2, ..., 
d(\zeta, f^m\xi) < \vp/2$. But $\xi \in \Lambda$ and $\Lambda$ is $f$-
invariant, hence
\begin{equation}\label{vecine}
d(\zeta, \Lambda) < \vp/2, ..., d(\zeta_{-m}, \Lambda) < \vp/2
\end{equation}

Let us denote by $J_m$ the set of these points $\zeta_{-m}(\hat y)$ obtained 
for all the prehistories $\hat y$ of points $y \in Y$. Relation \ref{vecine}, together with the fact that $\zeta \in \Lambda$ imply that $\zeta_{-m}(\hat y) \in \Lambda$, therefore $J_m \subset \Lambda$.
The relations in (~\ref{vecine}) imply also that $f^m$ is injective on a 
neighbourhood of $J_m$, since $\vp < d(\Lambda, \mathcal{C}_f)/4$ and 
$f^j(J_m) \cap \mathcal{C}_f = \emptyset, 
j = 0,...,m$. And, from our construction, $f^m(J_m) \subset W$. But from above
$f^m$ is injective on a neighbourhood of $J_m$ and it is bi-Lipschitz on
that neighbourhood, hence $HD(J_m) \le HD(W) = \delta^s(x)$. Recall also that
$t > \delta^s(x)$, so $t > HD(J_m)$.
This means that there exists $0< \gamma < \vp$, $\gamma$ small enough, 
and an open cover of $J_m$ with balls, $\mathcal{U} = (U_i)_{i \in I}$, such that  $\text{diam}U_i < \gamma$ and
\begin{equation}\label{Ui}
\sum\limits_{i \in I} (\text{diam}U_i)^t < \vp^{t+1} \cdot \lambda_s^{4n}\chi_s^n,
\end{equation}
for a fixed $n$, $n \ge n(\vp)$.

Let us choose now an arbitrary $i \in I$ and assume that $\text{Card}
(U_i \cap J_m) > 1$. 
Let us denote by $Y_i$ the set of points $y$ of $Y$ which have some 
prehistory $\hat y$ with $W^u_\vp(\hat y) \cap 
J_m \cap U_i \ne \emptyset$; denote by $F_i$ the set of prehistories
$\hat y \in \hat\Lambda$ with this property.

For each point $z' \in  U_i\cap J_m$, there exists then a point $y \in 
Y_i$ and a prehistory $\hat y \in \hat \Lambda$ such that $z' \in W^u_{\vp}
(\hat y)$, and actually $z' = \zeta_{-m}(\hat y)$. 
Therefore $z'$ has a prehistory $\hat z'$ given by that procedure, i.e
which is $\vp$-shadowed by $\hat y$; this prehistory may also be denoted
by $\hat z'(\hat y)$ if we want to emphasize its dependence on $\hat y$.
Let also $F'_i := \{\hat z'(\hat y), \hat y \in F_i\}$. 
Let us now take a prehistory $\hat z' \in F'_i$.
Since $\vp$ was assumed sufficiently small, we can define local branches of 
$f^{-1}$ on balls of radius $\vp$. Let us denote by $f^{-1}_*$ the branch 
of $f^{-1}$ defined on $B(z', \vp)$ such that $f^{-1}_*(z')=z'_{-1}$.
It may happen that the diameter of $f^{-1}_* U_i$ increases.
In case $\text{diam}f^{-1}_*U_i < \vp$, define afterwards the inverse 
iterate $f^{-2}_*$ such that $f^{-2}_*(z')
=z'_{-2}$, etc. Let us denote by $n_i(\hat z')$ the largest integer $n'$ 
which is 
a multiple of $n$ and for which $\text{diam} f^{-k'}_*(U_i) < \vp, 
0\le k' \le n'$, where $\hat z' = \hat z'(\hat y)$ for some $\hat y \in 
F_i \subset \hat\Lambda$ as above.
We do this for all the points of $U_i \cap J_m$ and denote by $n_i$ the 
largest integer $n_i(\hat z')$ for all $z' \in U_i\cap J_m$ and 
all prehistories $\hat z'$ from $F'_i$.
Obviously we cannot stretch the open set $U_i$ in backward time forever, 
while keeping the diameter of its inverse iterates smaller than $\vp$, hence 
$n_i$ is finite. Also, $n_i, n_i(\hat z')$ are multiples of $n$, 
so they can be written as $n_i = n m_i, n_i(\hat z') = n m_i(\hat z')$.
In addition, for a point $z'\in U_i \cap J_m$ and a prehistory 
$\hat z' \in F'_i$, 
we will define also the integer $\bar n_i(\hat z')$ as the smallest integer
(not necessarily a multiple of $n$) such that 
$\text{diam}f^{-\bar n_i(\hat z')}_* U_i > \vp$. 
We remark that the definitions imply the inequalities 
$$
n_i(\hat z') \le \bar n_i(\hat z') \le
n_i(\hat z') + n,
$$ 
for any point 
$z' \in J_m \cap U_i$ and any prehistory $\hat z' \in F'_i$.

We shall cover now the set $Y_i$ with sets of type 
$\Lambda(C', \vp)$, where $C' \in \mathcal{C}^n_*$ (i.e $C'$ are prehistories with respect to $f^n$).
In order to do this, take an arbitrary $z' \in \frac 12 U_i \cap J_m$ and a prehistory 
$\hat{z'}= \hat z'(\hat y) \in F'_i$, which corresponds to 
some complete (infinite) prehistory $C = \hat y \in F_i$. By $\frac 12 U_i$ we understand the ball with the same center as $U_i$ and with half its radius. 
Then consider the 
$m_i(\hat z')$-
prehistory $C'$ of $y$ (prehistory with respect to $f^n$), coming from the 
prehistory $C$, i.e we have $C' = (y, y_{-n},..., 
y_{-nm_i(\hat z')})$. Recall that $z' \in W^u_{\vp/2}(\hat y)$.
From the definition of $n_i(\hat{z'})$ we see immediately that 
$U_i \subset \bb P^2(C', \vp)$, and also $y \in \Lambda(C', \vp)$.
Recall that $C'$ is an $m_i(\hat z')$- prehistory with respect to $f^n$.
Hence, since $N_0(\vp)$ is the smallest cardinality of a cover
of $\Lambda$ with balls of radius $\vp$, and since $n_i= n m_i$ is the largest 
integer of the form $n_i(\hat z')$, we can cover the set $Y_i$ with at most
$N_0(\vp)^{m_i}$ sets of the form $\Lambda(C', \vp)$, where $C'$ are 
prehistories for $f^n$ of length $n(C')$, with $n(C') \le m_i$.
We will denote by $\Gamma_i$ the set of prehistories $C'$ used for
the last covering. So we have $Y_i \subset \mathop{\cup}
\limits_{C' \in \Gamma_i}\Lambda(C', \vp)$, and $\Gamma_i \subset 
\mathcal{C}^n_*, n(C') \le m_i, \forall C' \in \Gamma_i$. 
This construction can be done
for every $i \in I$ and, for each such $i$, we have 
Card$\Gamma_i \le N_0(\vp)^{m_i}$.

But we proved that, for all $\hat y \in \hat \Lambda$, the local unstable 
manifold $W^u_\vp(\hat y)$ intersects
$J_m$; on the other hand $J_m \subset 
\mathop{\cup}\limits_{i \in I}U_i$.
In conclusion, $Y \subset \mathop{\cup}\limits_{i \in I}Y_i$, hence 
$Y \subset \mathop{\cup}\limits_{i\in I}\mathop{\cup}
\limits_{C' \in \Gamma_i} \Lambda(C', \vp)$. 
Using this cover of $Y$ with sets $\Lambda(C', \vp), C' \in \mathcal{C}^n_*$,
we will estimate $M^-_{f^n}(0, (t+\eta)\phi^s_n, Y, N,\vp)$ for some
large integer $N$ chosen so that $n(C') \ge N, \forall C' \in \mathop{\cup}
\limits_{i \in I}\Gamma_i$:
$$
M^-_{f^n}(0, (t+\eta)\phi^s_n, Y, N, \vp)  \le \sum\limits_{i \in I}
\sum\limits_{C' \in 
\Gamma_i} \exp(S^-_{n(C')}(t+\eta)\phi^s_n(C'))
$$
Let us investigate now what is the relation between $\text{diam}U_i$
and $\exp(S^-_{n(C')}(t+\eta)\phi^s_n(C')), C' \in \Gamma_i$.
From the definition of $n_i(\hat z')$ we know that it represents the largest
integer $n'$, multiple of $n$, such that 
$\text{diam} f^{-k'}_*(U_i) < \vp, 0 \le k' \le n'$. 
Also, $\bar n_i(\hat z')$ represents
the smallest integer (not necessarily multiple of $n$) such that 
$\text{diam}f^{-\bar n_i(\hat z')}_* U_i > \vp$, where the inverse branches
$f^{-k}_*$ were defined along the prehistory $\hat z'= \hat z'(C)$. 

We consider now what happens to $U_i$ when taking inverse iterates. Let $z"$ be another point in $\frac 12 U_i \cap \Lambda$, and $\zeta "$ the intersection between $W^s_r(z")$ and the unstable manifold $W^u_r(\hat z')$; from the local product structure $\zeta" \in \Lambda$. Then, since $U_i$ is a ball, we get $\text{diam} f^{-\bar n_i(\hat z')} (W^s_r(z') \cap U_i) = constant \cdot |Df_s^{\bar n_i(\hat z')}(z'_{-\bar n_i(\hat z')}|^{-1}$, and $\text{diam} f^{-\bar n_i(\hat z')} (W^s_r(z") \cap U_i) = constant \cdot |Df_s^{\bar n_i (\hat z')}(\zeta"_{-\bar n_i(\hat z')})|$, due to the bounded distortion property from Proposition \ref{proC}.
But since $\zeta" \in W^u_r(\hat z')$ and $\hat \zeta"$ is the prehistory of $\zeta"$ following $\hat z'$, we see that the distance 
$d(z '_{-j}, \zeta "_{-j})$ decreases exponentially when j increases; thus due to the fact that $|Df_s|(z)$ depends Lipschitz continuously on $z$ (Theorem \ref{thm0}), we get that  $|Df_s^{\bar n_i (\hat z')}(\zeta"_{-\bar n_i(\hat z')})|$ and $|Df_s^{\bar n_i(\hat z')}(z'_{-\bar n_i(\hat z')}|$ are the same up to a constant independent of $z'$.

Therefore we will obtain, for every $i \in I$ that:
\begin{equation}\label{estim-diam}
\text{diam}U_i > \vp \exp(S^-_{\bar n_i(\hat z')} \phi^s(C'')) 
 \ge \vp \exp(S^-_{m_i(\hat z')} \phi^s_n(C'))
\lambda_s^n,
\end{equation}
where we considered first the $\bar n_i(\hat
z')$-prehistory $C'':= (y, y_{-1},..., y_{-\bar n_i(\hat z')})$,
(prehistory with respect to $f$, induced by the full prehistory $C:= \hat y$),
 and then
the $m_i(\hat z')$-prehistory $C':= (y, y_{-n},..., y_{-nm_i(\hat z')})$, 
(prehistory with respect to $f^n$, induced by the same complete prehistory 
$C$). We used also in (~\ref{estim-diam}) the fact that $\bar n_i(\hat z') \le n_i(\hat z') + n$.

Therefore by using (~\ref{estim-diam}) and the fact that
$\text{Card} \Gamma_i \le N_0(\vp)^{m_i}$, we can 
continue now with the estimate for $M^-_{f^n}(0, (t+\eta)\phi^s_n, Y, N, \vp)$
as follows:
\begin{equation}\label{M}
\aligned
M^-_{f^n}(0, (t+\eta)\phi^s_n, Y, N, \vp)& \le \sum\limits_{i\in I}
\sum\limits_{C' \in \Gamma_i} \vp^{-t-\eta} (\text{diam}U_i)^{t} \cdot exp(S^-_{m_i(\hat z')} \phi^s_n(C'))^\eta \lambda_s^{-n(t+\eta)} \\
& \le \sum \limits_{i\in I} [N_0(\vp)^{m_i} \cdot exp(S^-_{m_i(\hat z')} \phi^s_n(C'))^\eta] \vp^{-t-\eta} (\text{diam}U_i)^t  \lambda_s^{-n(t+\eta)}  \\
& \le \sum\limits_{i\in I} [N_0(\vp)\cdot \chi_s^{n\eta}]^{m_i} \chi_s^{-n\eta} \vp^{-t-\eta} (\text{diam}U_i)^t
 \lambda_s^{-n(t+\eta)} \\
\endaligned
\end{equation}
where we used in the last inequality the definition of $n_i(\hat z')$ and that $|Df_s^{n_i(\hat z')}(z'_{-n_i(\hat z')})|$ is the same as $|Df_s^{n_i(\hat z')}(z'_{-n_i(\hat z')})|$ up to a factor less than $\chi_s^n$ for any $z', z' \in U_i \cap J_m$. Thus we may as well use for $\hat z'$ the prehistory with the maximum $n_i(\hat z')$, hence with $n_i(\hat z') = n_i = n m_i$.

In the above sequence of inequalities, we used also that $0< \eta < 1$, $0 < t < 3$.
But $n_i = n m_i$, 
so (~\ref{M}) implies that 
\begin{equation}\label{M2}
\aligned
M^-_{f^n}(0, (t+\eta)\phi^s_n, Y, N, \vp)& \le \vp^{-t-1}\sum\limits_{i \in I}
 (\text{diam}U_i)^t  [N_0(\vp) \chi_s^{\eta n}]^{m_i} \lambda_s^{-4n}\chi_s^{-n}\\
&\le \vp^{-t-1}  \lambda_s^{-4n} \chi_s^{-n} \sum\limits_{i \in I}
 (\text{diam}U_i)^t  [N_0(\vp) \chi_s^{\eta n}]^{m_i}
\endaligned
\end{equation}

But from \ref{eta} and since $n \ge n(\vp, \eta)$, we see that
$N_0(\vp) \chi_s^{\eta n} < 1$.
From the way of
choosing the cover $\mathcal{U}$ in (~\ref{Ui}), 
we have also $\sum\limits_{i \in I} (\text{diam}U_i)^t < \vp^{t+1} \cdot \lambda_s^{4n}\chi_s^{n}$. In conclusion the inequality (~\ref{M2}) becomes

\begin{equation}\label{M3}
 M^-_{f^n}(0, (t+\eta)\phi^s_n, Y, N, \vp) < 1
\end{equation}

Since $\gamma$ and consequently $\text{diam}U_i, i \in I$ can be taken as small
as we wish, we see that $n(C')$ can also be made arbitrarily large, for 
$C' \in \mathop{\cup}\limits_{i \in I}\Gamma_i$. Therefore if $\gamma \to 0$, $N$ can be 
taken arbitrarily large, and (~\ref{M3}) implies that $M^-_{f^n}(0, (t
+\eta)\phi^s_n, Y, \vp) = 0$.
Thus one can conclude that $P^-_{f^n}((t+\eta)\phi^s_n, Y, \vp) \le 0$, for
$0< \eta <1$ and $n \ge n(\vp, \eta)$. But let us also remember that $Y$ was 
just the intersection between $\Lambda$ and one of the balls 
$B(y_1, \tilde\alpha\vp/4), ..., B(y_s, \tilde\alpha\vp/4)$ which cover
$\Lambda$.
Therefore by Proposition ~\ref{pro1} ii), it follows that 
$$
P^-_{f^n}((t+\eta)\phi^s_n, \Lambda, 
\vp) \le 0, \ \text{for} \ n \ge n(\vp, \eta).
$$
This implies that $t+\eta \ge t^s_{n}(\vp)$, for $n \ge n(\vp, \eta)$.
Since $t$ was chosen arbitrarily larger than $\delta^s(x)$, we obtain 
$\delta^s(x) + \eta \ge  t^s_{n}(\vp)$, for $n \ge n(\vp, \eta)$. 

\end{proof}

\begin{cor}\label{aplicatie}
In the same setting as in the previous Theorem, if $x, y$ are arbitrary
points from $\Lambda$, then $|\delta^s(x) - \delta^s(y)| \le 
\frac{(\ovl{\text{dim}}_B\Lambda) \cdot \log \chi_u}{\log \chi_s^{-1}}$, 
where $\chi_u := \sup\limits_{z \in \Lambda} |Df_u(z)|$.
\end{cor} 

\begin{proof}
First, let us notice that $\ovl{\text{dim}}_B \Lambda \le 4$ since 
$\Lambda \subset \mathbb P^2$, so even if $\ovl{\text{dim}}_B \Lambda$ 
cannot be 
calculated explicitly, the statement of the corollary still gives a good estimate of the maximum possible variation of $\delta^s(\cdot)$ on $\Lambda$.

Let us take an arbitrary $\eta$ with $\eta > \frac{(\ovl{\text{dim}}_B \Lambda) \log \chi_u}{\log \chi_s^{-1}}$ and an arbitrary $t$ with $t > \delta^s(x)$.
Then there exists $\beta_1 > \ovl{\text{dim}}_B \Lambda$ such that
$\eta > \frac{\beta_1 \cdot \log \chi_u}{\log \chi_s^{-1}}$.
Now, if $\beta_1 > \ovl{\text{dim}}_B \Lambda$, then there will exist a large integer $n_1= n_1(\beta_1)$ depending on $\beta_1$ such that for any $n \ge n_1$, $\rho_n$ is small enough so that $N_0(\rho_n) \le (\frac{1}{\rho_n})^{\beta_1}$, where $N_0(\cdot)$ and $\rho_n$ were introduced in the proof of Theorem \ref{thm1}.
Hence $N_0(\rho_n)\cdot \chi_s^{n\eta} \le 
(\vp\rho^n)^{-\beta_1}\chi_s^{n\eta}$.  
But we assumed $\eta > \frac{\beta_1 \log\chi_u}{\log\chi_s^{-1}}$, so there exists 
$n_1$ large enough and $\rho \in (0, \chi_u^{-1})$ close to $\chi_u^{-1}$, such that 
$(\vp\rho^n)^{-\beta_1}\chi_s^{n\eta} < 1$ for $n > n_1$.
This implies then: 
\begin{equation}\label{28}
N_0(\rho_n)\cdot \chi_s^{n\eta} < 1
\end{equation}
Now we can use inequality (\ref{28}) and (\ref{M2}) to prove that 
$M^-_{f^n}(0, (t+\eta)\phi^s_n, Y, \rho_n) < 1$; this implies then that
$$
P^-_{f^n}((t+\eta)\phi^s_n, \rho_n) \le 0, \text{for}\ n > n_1
$$
Thus we conclude from above that $t+ \eta \ge t^s_n(\rho_n)$.
But from Proposition \ref{pro4}, $t^s_n(\rho_n) = t^s, n \ge 1$. So $t+\eta \ge t^s$. Since $t$ is arbitrarily larger than $\delta^s(x)$ and $\eta$ is arbitrarily larger than $\frac{(\ovl{\text{dim}}_B \Lambda) \log \chi_u}{\log \chi_s^{-1}}$, it follows that $\delta^s(x)+\frac{(\ovl{\text{dim}}_B \Lambda) \log \chi_u}{\log \chi_s^{-1}} \ge t^s \ge \delta^s(y), y \in \Lambda$, \ where the inequality $t^s \ge \delta^s(y)$ follows from Theorem \ref{thm1}.
Therefore, $|\delta^s(x) - \delta^s(y)| \le 
\frac{(\ovl{\text{dim}}_B \Lambda)\cdot \log \chi_u}{\log \chi_s^{-1}}, 
\forall x, y \in \Lambda$.
\end{proof} 

\section{Independence of $\delta^s(x)$ when the map $f$ is open on $\Lambda$}

In this section we show that, for an Axiom A holomorphic map $f$ on $\mathbb P^2$ which, in addition, is also open on the basic set $\Lambda$, the stable dimension $\delta^s(x)$ becomes independent of $x \in \Lambda$.

It is easy to prove that for $\Lambda$ connected, the condition $f|_\Lambda: \Lambda \to \Lambda$ open, is equivalent to saying that the cardinality of the set $f^{-1}(x) \cap \Lambda$ is constant when $x$ ranges in $\Lambda$.

Fornaess and Sibony have introduced and studied in \cite{FS} a type of holomorphic maps $g$ on $\mathbb P^2$ which are Axiom A and such that the saddle part $S_1$ of the non-wandering set has a neighbourhood $U$ with the property that $g^{-1}(S_1) \cap U = S_1$ (among other properties). Such maps were called \textit{s-hyperbolic}. Notice that any s-hyperbolic map is in particular open on any 
basic set $\Lambda$ of saddle type. Examples of s-hyperbolic maps were given
in \cite{FS}.

In the sequel we will prove that the openness of $f$ on $\Lambda$ is a sufficient condition in order to guarantee that $\delta^s(x)$ does not depend on $x \in \Lambda$. The proof will use ideas and notations related to the concept of inverse pressure (the sets $\Lambda(C, \vp)$, and their concatenations, for example).

\begin{thm}\label{thm2}
Consider a holomorphic Axiom A map $f:\mathbb P^2 \to \mathbb P^2$ and a basic set of saddle type $\Lambda$ which does not intersect the critical set $\mathcal{C}_f$. Moreover assume that $f|_\Lambda : \Lambda \to \Lambda$ is open, in particular any point $x\in \Lambda$ has the same number of preimages in $\Lambda$ (this number being denoted by $d'$). 
Then for any $x \in \Lambda$, $\delta^s(x) = t^s_0$, where $t^s_0$ is the unique zero of the pressure function $t \to P(t\phi^s - \log d')$.
\end{thm}

\begin {proof}

In \cite{MU1}, we proved that $\delta^s(x) \le t^s_0$, so it remains to prove now only the opposite inequality. Denote $W:= W^s_r(x) \cap \Lambda$.
As in the second part of the proof of Theorem \ref{thm1}, we find an integer $m \ge 1$ and a set $J_m \subset f^{-m}W \cap \Lambda$ such that all local unstable manifolds of size $\vp/2$ intersect $J_m$ (for some small fixed $0<\vp<\vp_0$).
Take also $t > \delta^s(x)$ arbitrary. Then there exists a finite open cover $\mathcal{U}= (U_i)_{i \in I}$ of $J_m$ with balls of diameter less than $\gamma <<1$, and so that $\mathop{\sum}\limits_{i \in I}(\text{diam}U_i)^t < \frac 12$.
Recall from the proof of Theorem \ref{thm1} the definition of $F_i'$, the set of prehistories in $\Lambda$ of points from $U_i \cap J_m$. In the sequel, for 
the clarity of notation, we will denote the set $U_i\cap J_m$ by $U_i$ too.

Assume $\hat z$ is a prehistory in $\Lambda$ of a point $z \in U_i$; denote by $n(\hat z)$ the largest integer such 
that $\text{diam}f^{-k}_*U_i < \vp/2, 0 \le k \le n(\hat z)$, where $f^{-k}_*$ is the branch of $f^{-k}$ determined
by the prehistory $\hat z$. For the prehistory $\hat z$, denote by $C(\hat z)$ the $n(\hat z)$-prehistory 
$(z, z_{-1}, ..., z_{-n(\hat z)})$ which is obtained by truncating $\hat z$.

Now for each $i \in I$, let us fix a point $z_i \in \frac 12 U_i \cap \Lambda$ and then consider the set $\tilde F_i$ of all finite prehistories $C(\hat z_i)$ obtained as above, for all prehistories in $\Lambda$ of $z_i$. Notice that we consider in this case all $d’$ $f$-preimages in $\Lambda$ of a given point $z_i \in U_i$.

Denote also by $U_i^* := \mathop{\cup}\limits_{C \in \tilde F_i} \Lambda(C, \vp)$; then $\Lambda = \mathop{\cup}\limits_{i \in I} U_i^*$.
For later reference, it is useful to note that for any prehistory $\hat y \in \hat \Lambda$, there exists $j \in I$ such that $W^u_{\vp/2}(\hat y) \cap U_j \ne \emptyset$; but then there exists a certain prehistory $\hat z_j$ of $z_j$ such that $W^u_\vp (\hat y) \cap \Lambda \subset \Lambda(C(\hat z_j), \vp)$ (this follows from the definition of $C(\hat z_j)$ and the fact that $f|_\Lambda$ is open). Therefore, all unstable manifolds of prehistories in $\hat \Lambda$ (intersected with $\Lambda$ ) are contained in some $\Lambda(C, \vp), C \in \mathop{\cup}\limits_{i \in I} \tilde F_i$.

For $i \in I$, $C \in \tilde F_i$, write $C$ as $(z^C, ..., z^C_{-n(C)})$ (obviously notationally $z^C = z_i$). Denote also by 
$G_i:= \{n(C), C \in \tilde F_i\}$, (recall that $n(C)$ denotes the length of $C$), and write 
$G_i$ as $\{n_{i1}, ..., n_{iq_i}\}$, where $n_{i1} <..< n_{iq_i}$.
Now, let $N_{ij}$ be the number of prehistories $C \in \tilde F_i$ with $n(C) = n_{ij}, 1\le j \le q_i, i \in I$.

We will make the connection between the sets $\Lambda(C, \vp)$ (obtained as above in the process of covering $\Lambda$, in the definition of inverse pressure $P^-$), and the Bowen balls needed in the definition of the (forward) pressure.
In general by a \textit{Bowen ball} $B_k(z, \vp), z \in \Lambda$, we mean the set 
$\{ y \in \Lambda, d(f^j y, f^j z) < \vp, 0 \le j \le k \}$.
Therefore, if $C \in \tilde F_i, i \in I$, we have $\Lambda(C, \vp) = f^{n(C)}(B_{n(C)}(z^C_{-n(C)}, \vp))$; 
for simplicity of notation, denote the Bowen ball 
$B_{n(C)}(z^C_{-n(C)}, \vp)$ by $B(C), C \in \tilde F_i, i \in I$.
From the above discussion, we know that $\Lambda = \mathop{\cup}\limits_{i \in I}U_i^* = \mathop{\cup}\limits_{i \in I}\mathop{\cup}\limits_{C \in \tilde F_i} f^{n(C)}(B(C))$.
However since the integers $n(C)$ are different among themselves, it does not follow directly that the Bowen balls $B(C)$ cover $\Lambda$. In order to get a covering of $\Lambda$ with Bowen balls, we will make a construction using 
concatenations of sets of type $\Lambda(C, \vp)$; it will be possible then to take the lengths of these 
concatenations arbitrarily large.  

Let in general $C$ and $C'$ be two prehistories of points in $\Lambda$, $C = (z, z_{-1}, ..., z_{-n(C)})$ and 
$C' = (w, w_{-1}, ..., w_{-n(C)})$. Assume also that there exists a point $ z' \in \Lambda(C, \vp)$, so that 
$z'_{-n(C)} \in \Lambda(C', \vp)$, where $z'_{-n(C)}$ represents the $n(C)$-preimage of $z'$ which is 
$\vp$-shadowed by $z_{-n(C)}$.
If $z'_{-n(C)} \in \Lambda(C', \vp)$, it follows that it has a prehistory $(z'_{-(n(C)+1)},..., z'_{-(n(C)+n(C'))})$
which is $\vp$-shadowed by $C'$.
So we can form the set $\Lambda(CC', \vp):= \{y \in \Lambda(C, \vp), y_{-n(C)} \in \Lambda(C', \vp)\}$, and from above, 
if this set is non-empty, then $\Lambda(CC', \vp) \subset \Lambda(C'', 2\vp)$, 
where $C''$ is an $(n(C)+n(C'))$-prehistory. This process will be called \textit{concatenation}.
 
We will use concatenation repeatedly in order to obtain a cover of $\Lambda$ with sets 
$\Lambda(C'', 2\vp)$ with $n(C'')$ arbitrarily large.
Define now the collection $\Gamma_n := \{\bar C = C_1...C_s, C_k \in \tilde F_{j_k}, j_k \in I, 1\le k \le s,
n \le n(C_1)+...+n(C_{j_s}) < n+N\}$, where here $N:= \mathop{\text{max}}\limits_{i \in I, C \in \tilde F_i} n(C)$.
Since $\Lambda = \mathop{\cup}\limits_{i \in I}\mathop{\cup}\limits_{C \in \tilde F_i} \Lambda(C, \vp)$, we see that also $$
\Lambda = \mathop{\cup}\limits_{\bar C \in \Gamma_n} \Lambda(\bar C, 2\vp)
$$
If $\bar C \in \Gamma_n$, and $\bar C = C_1...C_s$, denote by $n(\bar C) := n(C_1)+...+n(C_s)$.
But as noticed before, if $\bar C \in \Gamma_n$, there exist points $z^{\bar C}_{-n(\bar C)}$ such that
$\Lambda(\bar C, 2\vp) = f^{n(\bar C)} (B_{n(\bar C)}(z^{\bar C}_{-n(\bar C)}, 2\vp))$, 
and 
$n \le n(\bar C) < n+N$. Therefore $\Lambda = \mathop{\cup}\limits_{\bar C \in \Gamma_n}
f^n(f^{n(\bar C)-n} B_{n(\bar C)}(z^{\bar C}_{-n(\bar C)}, 2\vp))$.

Let us recall now the remark made earlier, after the definition of $U_i^*$.
Since any set  $W^u_{\vp/2}(\hat y)\cap \Lambda, \hat y \in \hat \Lambda$ is contained in $\Lambda(C, \vp)$ 
for some $C \in \mathop{\cup}\limits_{i \in I} \tilde F_i$ and since we collected the corresponding $C(\hat z_i)$ for all prehistories $\hat z_i \in \hat\Lambda$ and all $i \in I$, we obtain that any $f^n$-preimage in $\Lambda$ of a point from $\Lambda$ belongs to the union $\mathop{\cup}\limits_{\bar C \in \Gamma_n} f^{n(\bar C)-n} B_{n(\bar C)}(z^{\bar C}_{-n(\bar C)}, 2\vp))$. 
So we can conclude that $\Lambda = \mathop{\cup}\limits_{\bar C \in \Gamma_n}
f^{n(\bar C)-n} B_{n(\bar C)}(z^{\bar C}_{-n(\bar C)}, 2\vp))$.

On the other hand, notice that $f^{n(\bar C)-n} B_{n(\bar C)}(z^{\bar C}_{-n(\bar C)}, 2\vp) \subset B_n(z^{\bar C}_{-n}, 2\vp)$.

Denote then $F_n:= \{z^{\bar C}_{-n}, \bar C \in \Gamma_n\}$. From the previous considerations it follows that
$F_n$ is an $(n, \vp)$-spanning set for $\Lambda$, in the classical (forward) sense.
We will use this particular spanning set $F_n$ in order to estimate
$$
P_n(t\phi^s-\log d') := \inf\{\sum\limits_{z\in F} 
e^{S_n(t\phi^s)(z)- n \log d'}, \ F (n, \vp)-\text{spanning set for} \ \Lambda\}
$$
Let us remember the construction of the set $F_n$ and the points $z^{\bar C}_{-n(\bar C)}$.
If $\bar C = C_1...C_s, C_k \in \tilde F_{j_k}, 1 \le k \le s$, then from the proof of Proposition \ref{proC}, 
we have 
that there exists a positive constant $\sigma$ so that 
$|Df_s^{n(C_s)}(z^{\bar C}_{-n(\bar C)})| \le e^{\sigma\vp}\cdot \text{diam}U_{j_s}, 
..., |Df_s^{n(C_1)}(z^{\bar C}_{-n(C_1)})| \le e^{\sigma\vp}\cdot \text{diam}U_{j_1}$, \ (since $C_1\in 
\tilde F_{j_1}, ..., C_s \in \tilde F_{j_s}$).
Hence since $n \le n(\bar C) < n+N$, there will exist a positive constant $T_1$ independent of $n$ such that
$|Df_s^{n(\bar C)}(z^{\bar C}_{-n(\bar C)})| \le T_1 \cdot e^{n\sigma\vp} \cdot 
(\text{diam}U_{j_1})\cdot ...\cdot 
(\text{diam} U_{j_s})$. But recall that $|Df_s^{n(\bar C)}(z^{\bar C}_{-n(\bar C)})| = 
|Df_s^{n(\bar C)-n}(z^{\bar C}_{-n(\bar C)})|\cdot
|Df_s^n(z^{\bar C}_{-n})|$. Thus, for a positive constant $T_2$ we obtain the inequality:

\begin{equation}\label{diametru}
|Df_s^n (z^{\bar C}_{-n})| \le T_2 \cdot e^{n\sigma\vp}\cdot (\text{diam}U_{j_1})\cdot ... \cdot 
(\text{diam}U_{j_s}), 
\end{equation}
for all $\bar C \in \Gamma_n$ and all integers $n > 1$.

Now given $n$, and $j_1, ... j_s \in I$, we will estimate how many prehistories $\bar C = C_1...C_s$ there exist, 
with
 $C_k \in \tilde F_{j_k}, 1 \le k \le s$ and $\bar C \in \Gamma_n$.

For $i \in I$ and $1 \le j \le q_i$, we denoted by $N_{ij}$ the number of prehistories 
$C \in \tilde F_i$ with $n(C) = n_{ij}, \ n_{ij} \in G_i$.
Hence for each $s$, $j_1, ..., j_s \in I$, and integers $n_{j_k p_k} \in G_{j_k}, 1\le k \le s$, satisfying 
$n \le n_{j_1 p_1}+...+n_{j_s p_s}< n+N$, there exist at most $N_{j_1 p_1}\cdot ... \cdot N_{j_s p_s}$ prehistories
of type $\bar C = C_1...C_s$ in $\Gamma_n$ with $C_k \in \tilde F_{j_k}$ and 
$n(C_k) = n_{j_k p_k}, 1 \le k \le s$.
If $i \in I$, denote by 
$$
\Sigma_i:= \frac{N_{i1}}{(d')^{n_{i1}}} + ... + \frac{N_{iq_i}}{(d')^{n_{iq_i}}}
$$
To start with, let us compare $N_{i1}$ and $N_{i2}$. Since $n_{i1} < n_{i2}$, the prehistories stopping at 
$n_{i1}$ cannot be continued to $n_{i2}$-prehistories; hence using the fact that each point in 
$\Lambda$ has at most $d'$ preimages in $\Lambda$, 
it follows that $N_{i2} \le [(d')^{n_{i1}}-N_{i1}]\cdot (d')^{n_{i2}-n_{i1}}$. 
Similarly one can show that $N_{ij} \le (d')^{n_{ij}}-N_{i1} (d')^{n_{ij}-n_{i1}} - ...- 
N_{i(j-1)} (d')^{n_{ij}-n_{i(j-1)}}, 2\le j \le q_i$. 
This implies that, for each $i \in I$, we obtain:
\begin{equation}\label{sigmai}
\aligned
\Sigma_i \le & \frac{N_{i1}}{(d')^{n_{i1}}} + \frac{N_{i2}}{(d')^{n_{i2}}}+...\frac{N_{i(q_i-1)}}{(d')^{n_{i(q_i-1)}}} + 
 \frac{(d')^{n_{iq_i}}-N_{i1}(d')^{n_{iq_i}-n_{i1}}-...-N_{i(q_i-1)}(d')^{n_{iq_i}- n_{i(q_i-1)}}}{(d')^{n_{iq_i}}} \\
& \le (1-\frac{N_{i1}}{(d')^{n_{i1}}} -...- \frac{N_{i(q_i-1)}}{(d')^{n_{i(q_i-1)}}}) 
+ \frac{N_{i1}}{(d')^{n_{i1}}} +...+ \frac{N_{i(q_i-1)}}{(d')^{n_{i(q_i-1)}}} = 1
\endaligned
\end{equation}

Therefore from the last inequality it follows that $\Sigma_i \le 1, i \in I$ and hence
 $\Sigma_{j_1}\cdot ... \cdot \Sigma_{j_s} \le 1, j_1, ..., j_s \in I$. This implies then 
$\sum\limits_{1\le p_1\le q_{j_1}, ..., 1\le p_s \le q_{j_s}} \frac{N_{j_1p_1}...N_{j_s p_s}}{(d')^{n_{j_1 p_1}+...+n_{j_s p_s}}} \le 1$.
In particular,  if $j_1, ..., j_s \in I$, we get 
\begin{equation}\label{da}
\sum' \frac{N_{j_1 p_1}\cdot ... \cdot N_{j_s p_s}}{(d')^n} \le \Theta,
\end{equation} 
where $\Theta >0$ is a constant independent of $n, j_1, ..., j_s$ and
where the sum $\mathop{\sum}\limits^{'}$ is taken over all integers $n_{j_k p_k} \in G_{j_k}, 1\le k \le s$ satisfying 
$n \le n_{j_1 p_1}+...+n_{j_s p_s}< n+N$. 

We will use the above conclusions in order to estimate now $\sum\limits_{z\in F_n} 
e^{S_n(t\phi^s)(z)- n \log d'}$; first notice that for each $j_1, .., j_s \in I$, there exist 
at most  $\sum N_{j_1 p_1}\cdot ... \cdot N_{j_s p_s}$  prehistories 
$\bar C = C_1...C_s \in \Gamma_n$, with $C_k \in \tilde F_{j_k}, 1 \le k \le s$,
where the last sum is taken over all integers $n_{j_k p_k} \in G_{j_k}, 1\le k \le s$ 
satisfying  $n \le n_{j_1 p_1}+...+n_{j_s p_s}< n+N$. 
Then using (\ref{diametru}) and (\ref{da}), we will obtain: 
\begin{equation}\label{gata}
\aligned
P_n(t\phi^s-\log d') &\le \sum\limits_{z\in F_n} e^{S_n(t\phi^s)(z)-n \log d'}\\
& \le \sum'' (\sum’ N_{j_1 p_1}\cdot ... \cdot N_{j_s p_s}) \cdot (d')^{-n}
\cdot T_2 
e^{n\sigma\vp}\cdot (\text{diam}U_{j_1})^t\cdot ... \cdot (\text{diam}U_{j_s})^t \\
& \le \Theta T_2 \cdot e^{n\sigma\vp} \cdot \sum'' (\text{diam}U_{j_1})^t\cdot ... 
\cdot (\text{diam}U_{j_s})^t,
\endaligned
\end{equation}
where the sum $\mathop{\sum}\limits^{''}$ is taken over all integers $s >0$ and $s$-uples 
$j_1, ..., j_s \in I$ having some prehistories $C_1, ..., C_s$ in 
$\tilde F_{j_1}, ..., \tilde F_{j_s}$ respectively, which satisfy:  $C_1...C_s \in \Gamma_n$.
But the cover $(U_i)_{i \in I}$ has been taken such that 
$\sum\limits_{i \in I}(\text{diam}U_i)^t < \frac 12$, therefore 
$\sum\limits_{s>0}(\sum\limits_{i \in I}(\text{diam}U_i)^t)^s < 1$.
This implies that $\sum\limits_{s >0}\sum\limits_{j_1, ..., j_s \in I} (\text{diam}U_{j_1})^t\cdot...\cdot
 (\text{diam}U_{j_s})^t < 1$.
Therefore using (\ref{gata}) it follows that 
$$
P_n(t\phi^s - \log d') <  \Theta T_2 \cdot e^{n\sigma\vp} 
$$
The constants $\Theta, T_2, \sigma$ do not depend on $n, \vp$, if $\vp < \vp_1$ is small enough.
So we get $P(t\phi^s-\log d') = \mathop{\ovl{\lim}}\limits_n \frac{1}{n} \log P_n \le \sigma \vp$, and since 
$\vp>0$ is arbitrarily small, we get $P(t\phi^s-\log d') \le 0$. But this means that $t \ge t^s_0$, where $t^s_0$ 
denotes the unique zero of the function $t \to P(t\phi^s - \log d')$. Now recall that $t$ has been taken 
arbitrarily larger than $\delta^s(x)$, hence $\delta^s(x) \ge t^s_0$.
Recalling that the opposite inequality was proved in \cite{MU1}, 
we get finally that $\delta^s(x) = t^s_0, x \in \Lambda$. So, in case $f|_\Lambda$ is open, 
the stable dimension is independent of the point.
\end{proof}

In particular Theorem \ref{thm2} shows that in the case of 
s-hyperbolic maps studied in \cite{FS}, the stable dimension along basic sets of saddle type, is  
independent of the point .

Finally, notice that the proof of Theorem \ref{thm2} shows more 
generally that $\delta^s(x) \ge t^s_0$ if each point of $\Lambda$ has at 
most $d'$-preimages in $\Lambda$ (one may also denote $t^s_0$ by $t^s_0(d')$ when emphasizing its
 dependence on $d'$).
The number of preimages $d(x)$ that a point $x$ from $\Lambda$ has in $\Lambda$, is not necessarily constant. The above remark and Theorem 1.2 of \cite{MU1} prove the following:

\begin{cor}\label{ultimul}
In the setting of Theorem \ref{thm1}, if $d' \le  d(y) \le d'', y \in \Lambda$, then for each $x\in \Lambda$ it follows that $ t^s_0(d'') \le \delta^s(x) \le  t^s_0(d')$.
\end{cor}

It is important to remark that this Corollary does not require $f|_\Lambda$ to be open; it gives estimates of the stable dimension, for example in the 
case of quadratic maps from \cite{MU1}.

\section{Results in the real conformal case}

Most of the results of the previous sections work also in a more general setting, although for historical and example reasons we preffered to give them in the holomorphic case.

\begin{defn}
Let $M$ be a compact Riemannian manifold of real dimension 4, and 
$f:M \to M$ a $\mathcal{C}^r, r \ge 2$ map, possibly non-invertible. Let also $\Lambda$ a \textit{basic set of saddle type} for $f$, i.e there exists an open neighbourhood $V$ of $\Lambda$ in $M$, such that $\Lambda = \mathop{\cap}\limits_{n \in \mathbb Z} f^n(V)$, $f|_{\Lambda}:\Lambda \to \Lambda$ is transitive and $f$ is hyperbolic on $\hat \Lambda$ with both expanding and contracting directions.
Suppose also that $f$ is finite-to-one, the dimension of stable tangent spaces on $\Lambda$ is 2, and $f$ is conformal on its stable manifolds on $\Lambda$. 
We will say that such a map $f$ is \textbf{c-hyperbolic} on the basic set $\Lambda$. ("c" coming from "conformal").$\hfill\square$
\end{defn}

The notations for the stable dimension $\delta^s(x)$, the zero of the inverse pressure $t^s_n(\vp), t^s_n$, etc., remain the same.

The following theorems are proved in the same way as the previous corresponding theorems in the holomorphic case.

\begin{thm}
Consider $f:M \to M$ a c-hyperbolic map on the basic set $\Lambda$, such that $\mathcal{C}_f \cap \Lambda = \emptyset$. Then the map $x \to E^s_x$ is Lipschitz continuous and in particular, if $\phi^s(y):= \log |Df_s(y)|, y \in \Lambda$, then $\phi^s$ is Lipschitz continuous on $\Lambda$.
\end{thm}

\begin{thm}\label{treal}

Let $f:M \to M$ be a c-hyperbolic map on a basic set $\Lambda$, with 
$\mathcal{C}_f \cap \Lambda = \emptyset$. Then:

(a) for every $x \in \Lambda$, we have $\delta^s(x) \le t^s_n(\rho_n) = t^s$, 
where $\rho_n$ are numbers of the form $\vp \rho^n, n \ge 1$, with $\chi_u:= \mathop{\sup}\limits_{\Lambda}||Df_u||$, and $\rho > 0$ is an arbitrary number smaller than $\chi_u^{-1}$, and $\vp < min\{\vp_0, r_0\}$.

(b) for all positive numbers $\vp < \vp_0$, and $\eta > 0$, we obtain $\delta^s(x) + \eta \ge t^s_n(\vp)$, where $n \ge n(\vp, \eta) > \frac{4 \log (\vp^{-1})}{\eta \log \chi_s^{-1}}$. In particular, if $\eta = \vp$ small enough, we get $\delta^s(x) + \vp \ge t^s_n(\vp)$, for $n > \frac{1}{\vp^{1.1}}$.
\end{thm}

Consequently we have the similar corollary:

\begin{cor}
In the same setting as in Theorem \ref{treal}, if $x, y \in \Lambda$, then $|\delta^s(x) - \delta^s(y)| \le \frac{(\ovl{dim}_B \Lambda)\cdot \log \chi_u}{\log \chi_s^{-1}}$.
\end{cor}

\begin{thm}\label{sf}
Consider a smooth map $f: M \to M$ which is c-hyperbolic on a connected 
basic set $\Lambda$ which does not intersect the critical set $\mathcal{C}_f$. Moreover assume that $f|_{\Lambda} : \Lambda \to \Lambda$ is open, in particular any point $x \in \Lambda$ has the same number of preimages in $\Lambda$ (denote this number by $d'$). Then for any $x \in \Lambda, \delta^s(x) = t^s_0$, where $t^s_0$ is the unique zero of the pressure function $t \to P(t\phi^s - \log d')$.
\end{thm}

\begin{cor}
Let $f:M \to M$ be a smooth map, c-hyperbolic on the basic set $\Lambda$ which does not intersect the critical set $\mathcal{C}_f$; if $d(y)$ denotes the cardinality of $f^{-1}(y) \cap \Lambda, y \in \Lambda$ and $d', d''$ are positive integers such that $d' \le d(y) \le d'', \forall y \in \Lambda$, then for each $x \in \Lambda$ it follows that $t^s_0(d'') \le \delta^s(x) \le t^s_0(d')$, where $t^s_0(d')$ represents the unique zero of the pressure function $t \to P(t\phi^s - \log d')$.
\end{cor}

This corollary does not require $f|_{\Lambda}$ to be open.

\medskip

The authors are grateful to the referee for a careful reading of the paper, several comments, and for suggesting to include a separate section with the real case.

\bigskip

Eugen Mihailescu: \ 
Institute of Mathematics ``Simion Stoilow'' of the Romanian Academy, 
P. O. Box 1-764, 
RO 70700, Bucharest, Romania.
Email: Eugen.Mihailescu\@@imar.ro

Webpage: http://stoilow.imar.ro/~mihailes/

\

Mariusz Urba\'nski: \ 
Department of Mathematics, University of North Texas, 
P.O. Box 311430, Denton, TX 76203-1430, USA.
Email:urbanski\@@unt.edu 

Webpage: www.math.unt.edu/~urbanski

\end{document}